\documentclass{amsart}
\usepackage{amssymb}
\usepackage{hyperref}

\newtheorem{thm}{Theorem}[section]
\newtheorem{cor}[thm]{Corollary}
\newtheorem{lem}[thm]{Lemma}
\newtheorem{prop}[thm]{Proposition}
\theoremstyle{definition}
\newtheorem{defn}[thm]{Definition}
\theoremstyle{remark}

\numberwithin{equation}{section}
\numberwithin{thm}{section}

\newcommand{\C}{{\mathbb{C}}}
\newcommand{\R}{{\mathbb{R}}}
\newcommand{\N}{\mathbb N}

\DeclareMathOperator*{\diam}{diam}
\DeclareMathOperator*{\dist}{dist}
\newcommand{\eps}{{\varepsilon}}
\newcommand{\propagateomega}{{e^{it\Delta_{\Omega}}}}

\newcommand{\lsm}{\lesssim}

\newcommand{\ld}{\Delta_{\Omega}}
\newcommand{\po}{P^{\Omega}}

\newcommand{\lpo}{L^p(\Omega)}

\newcommand{\hoo}{H^1_0(\Omega)}

\newcommand{\hr}{ H^1(\R^3)}
\newcommand{\lt}{L_{t,x}^{5}(\R\times\Omega)}

\newcommand{\prdn}{e^{it_n\ld}}
\newcommand{\prdnn}{e^{-it_n\ld}}
\newcommand{\ro}{\R\times\Omega}

\newcommand{\on}{\Omega_n}

\newcommand{\qtq}[1]{\quad\text{#1}\quad}

\let\Re=\undefined\DeclareMathOperator*{\Re}{Re}
\let\Im=\undefined\DeclareMathOperator*{\Im}{Im}

\newcounter{smalllist}

\newenvironment{CI}{\begin{list}{{\ $\bullet$\ }}{%
\setlength{\topsep}{0mm}\setlength{\parsep}{0mm}\setlength{\itemsep}{0mm}%
\setlength{\labelwidth}{0mm}\setlength{\itemindent}{-1.5em}\setlength{\leftmargin}{1.5em}%
\setlength{\labelsep}{0mm} }}{\end{list}}

\title{The focusing cubic NLS on exterior domains in three dimensions}

\author[R. Killip]{Rowan Killip}
\address{Department of Mathematics, University of California, Los Angeles, CA 90095}%
\email{killip@math.ucla.edu}

\author[M. Visan]{Monica Visan}
\address{Department of Mathematics, University of California, Los Angeles, CA 90095}
\email{visan@math.ucla.edu}

\author[X. Zhang]{Xiaoyi Zhang}
\address{Department of Mathematics, University of Iowa, Iowa City, IA 52242 and the
Chinese Academy of Science, Beijing}%
\email{xiaoyi-zhang@uiowa.edu}
{\normalsize }

\begin{document}
\maketitle
\begin{abstract}
We consider the focusing cubic NLS in the exterior $\Omega$ of a smooth, compact, strictly convex obstacle in three dimensions. We prove that the threshold for global existence and scattering is the same as for the problem posed on Euclidean space. Specifically, we prove that if $E(u_0)M(u_0)<E(Q)M(Q)$ and $\|\nabla u_0\|_2\|u_0\|_2<\|\nabla Q\|_2\|Q\|_2$, the corresponding solution to the initial-value problem with Dirichlet boundary conditions exists globally and scatters to linear evolutions asymptotically in the future and in the past. Here, $Q(x)$ denotes the ground state for the focusing cubic NLS in $\R^3$.
\end{abstract}

\section{Introduction}
We consider the focusing cubic nonlinear Schr\"odinger equation in the exterior of a strictly convex obstacle $\Omega^c\subset\R^3$ with Dirichlet boundary
conditions
\begin{align}\label{nls}(\text{NLS}_\Omega)
\begin{cases}
i u_t+\Delta_\Omega u=-|u|^2u \\
u(0,x)=u_0(x).
\end{cases}
\end{align}
Here $u:\R\times \Omega\to \C$ and $-\Delta_\Omega$ denotes the Dirichlet Laplacian, which is a self-adjoint operator on $L^2(\Omega)$ with form domain $H^1_0(\Omega)$.  We take initial data $u_0\in H^1_0(\Omega)$.

When posed on the whole Euclidean space $\R^3$, the problem is scale invariant.  More precisely, the mapping
\begin{align}\label{scaling}
u(t,x)\mapsto u^{\mu}(t,x):=\mu \:\! u(\mu^2 t, \mu x) \qtq{with} \mu>0
\end{align}
leaves the class of solutions to $\text{NLS}_{\R^3}$ invariant.  This scaling also identifies the space $\dot H_x^{1/2}$ as the critical space.  Since the presence of the obstacle does not change the intrinsic dimensionality of the problem, we may regard equation \eqref{nls} is being subcritical for data in $H^1_0(\Omega)$.

Throughout this paper we restrict ourselves to the following notion of solution:

\begin{defn}[Solution] Let $I$ be a time interval containing the origin. A function $u:I\times\Omega\to \C$ is called a (strong) solution to \eqref{nls} if it lies in the class $C_t(I', H^1_0(\Omega))\cap L_{t,x}^5(I'\times\Omega)$ for any compact interval $I'\subset I$, and it satisfies the Duhamel formula
\begin{align}\label{duhamel}
u(t)=e^{it\Delta_\Omega}u_0+i\int_0^t e^{i(t-s)\Delta_\Omega}(|u|^2u)(s)\,ds
\end{align}
for all $t\in I$.  The interval $I$ is said to be maximal if the solution cannot be extended beyond $I$. We say $u$ is a global solution if $I=\R$.
\end{defn}

In this formulation, the Dirichlet boundary condition is enforced through the appearance of the linear propagator associated to the Dirichlet Laplacian.

Local well-posedness for the defocusing version of \eqref{nls} was established in several works; see, for example, \cite{Anton08, bss:schrodinger, iva:convex, PlanchonVega}.  Standard arguments relying on the Strichartz estimates proved in \cite{iva:convex} and the equivalence of Sobolev spaces proved in \cite{kvz:harmonic} (see Theorem~\ref{T:equiv}) can be used to construct a local theory for \eqref{nls}.

\begin{thm}[Local well-posedness]\label{T:lwp} Let $u_0\in H^1_0(\Omega)$.  Then there exists a maximal time interval $I$ such that \eqref{nls} admits a strong solution on $I$.  The interval $I$ and the solution $u$ are uniquely determined by $u_0$.  Moreover, the following hold:
\begin{CI}
\item $I$ is an open interval containing the origin.
\item $($Conservation laws$)$  Mass and energy are conserved by the flow: for any $t\in I$,
\begin{align*}
M(u(t))&:=\int_{\Omega} |u(t,x)|^2 \,dx=M(u_0),\\
E(u(t))&:=\int_{\Omega}\tfrac12|\nabla u(t,x)|^2 -\tfrac 14|u(t,x)|^4 \,dx=E(u_0).
\end{align*}
\item $($Blowup$)$ If $\sup I<\infty$, then $\limsup_{t\to \sup I}\|u(t)\|_{H^1_0(\Omega)}=\infty$.
\item $($Scattering$)$ Suppose $[0, \infty)\subseteq I$ and $\|u\|_{L_{t,x}^5([0, \infty)\times\Omega)}<\infty$.  Then $u$ scatters forward in time, that is, there exists $u_+\in H^1_0(\Omega)$ such that
\begin{align*}
\|u(t)-e^{it\Delta_\Omega}u_+\|_{H^1_0(\Omega)}\to 0 \qtq{as} t\to \infty.
\end{align*}
\item $($Small data GWP$)$ There exists $\eta_0>0$ such that if $\|u_0\|_{H^1_0(\Omega)}\leq \eta_0$, then the solution $u$ to \eqref{nls} is global and satisfies
$$
\|u\|_{L_{t,x}^5(\R\times\Omega)}\lesssim \|u_0\|_{H^1_0(\Omega)}.
$$
In particular, the solution scatters in both time directions.
\end{CI}
\end{thm}

In this paper, we consider the global existence and scattering question for large initial data.  To put the problem in context, let us first recall some earlier results for the equivalent problem posed in the whole Euclidean space $\R^3$.  In \cite{dhr} it is shown that the focusing cubic NLS on $\R^3$ is globally well-posed and scatters whenever the initial data lies below the ground state threshold.  To state this result explicitly, let $Q$ denote the unique, positive, spherically-symmetric, decaying solution to the elliptic problem
\begin{align*}
\Delta Q-Q+Q^3=0.
\end{align*}
The main result in \cite{dhr} states that whenever
\begin{align}\label{dhr hyp}
E(u_0)M(u_0)<E(Q)M(Q) \qtq{and} \|\nabla u_0\|_{L^2_x}\|u_0\|_{L^2_x}<\|\nabla Q\|_{L^2_x}\|Q\|_{L^2_x},
\end{align}
the solution to $\text{NLS}_{\R^3}$ is global and it satisfies global $L_{t,x}^5$ spacetime bounds.  For spherically-symmetric initial data, this result was proved in the earlier work \cite{hr:radial}.

Note that in the Euclidean case, the quantity $E(u)M(u)$ is scale invariant and conserved in time.  Throughout this paper, we will refer to this quantity as the \emph{$H^{1/2}_x$-energy}.  Note that both of the assumptions on the initial data in \eqref{dhr hyp} are scale invariant.

By the variational characterization of $Q$, we know that the only functions $f\in H^1(\R^3)$ that satisfy
$$
E(f)M(f)=E(Q)M(Q) \qtq{and} \|\nabla f\|_{L^2(\R^3)}\|f\|_{L^2(\R^3)}=\|\nabla Q\|_{L^2(\R^3)}\|Q\|_{L^2(\R^3)}
$$
are of the form $f(x)= e^{i\theta}\rho Q(\rho [x-x_0])$ where $\theta\in[0,2\pi)$, $\rho\in (0,\infty)$, and $x_0\in \R^3$.  In the presence of an obstacle, \emph{there are no functions} $f\in H^1_0(\Omega)$ such that
$$
E(f)M(f)=E(Q)M(Q) \qtq{and} \|\nabla f\|_{L^2(\Omega)}\|f\|_{L^2(\Omega)}=\|\nabla Q\|_{L^2(\R^3)}\|Q\|_{L^2(\R^3)}.
$$
This can be seen easily by extending $f$ to be identically equal to zero on the obstacle.  The variational characterization of $Q$ on $\R^3$ then yields that $f$ must be $Q$ up to the symmetries of the equation; these functions, however, do not obey Dirichlet boundary conditions.

We contend that even though \eqref{nls} does not admit a direct analogue of the soliton $Q$, the threshold for global well-posedness and scattering is still the same as for the problem posed on $\R^3$.  The main result of this paper verifies the positive part of this claim:

\begin{thm}[Global well-posedness and scattering]\label{T:main} Let $u_0\in H^1_0(\Omega)$ satisfy
\begin{align}
E(u_0)M(u_0)&<E(Q)M(Q),\label{em}\\
\|\nabla u_0\|_{L^2(\Omega)}\|u_0\|_{L^2(\Omega)}&<\|\nabla Q\|_{L^2(\R^3)}\|Q\|_{L^2(\R^3)}.\label{km}
\end{align}
Then there exists a unique global solution $u$ to \eqref{nls} and
\begin{align}\label{E:STB}
\|u\|_{L_{t,x}^5(\R\times\Omega)} \leq C\bigl(M(u_0), E(u_0)\bigr).
\end{align}
In particular, $u$ scatters in both time directions.
\end{thm}

Observe that our bound in \eqref{E:STB} depends jointly on the mass and energy of the initial data, not simply their product.  A bound depending only on $M(u_0)E(u_0)$  could be obtained (with significant additional complexity) by incorporating rescaling into our arguments in the style of \cite{kvz:quintic};  however, we know of no application that benefits from this stronger assertion.

Due to the convexity of the curve $ME=E(Q)M(Q)$, we may exhaust the region of the mass/energy plane where \eqref{em} holds by sub-level sets of the following one-parameter family of free energies: for each $0<\lambda<\infty$, we define
\begin{align*}
F^\lambda(u_0) :=E(u_0) + \lambda M(u_0).
\end{align*}
To be precise, setting
$$
F^\lambda_*:=2\sqrt{\lambda M(Q)E(Q)}
$$
one easily sees that
\begin{align}\label{Exhaust A}
\{ u_0 : E(u_0)M(u_0)<E(Q)M(Q) \} = \bigcup_{0<\lambda<\infty} \{ u_0 : F^\lambda (u_0)< F^\lambda_* \}
\end{align}

Note that $F^\lambda_*= F^\lambda( Q^\mu )$ where $\mu = \sqrt{\lambda M(Q)/ E(Q)}$ and $Q^\mu$ is the rescaling of $Q$ defined via \eqref{scaling}.  Correspondingly, we note that $F^\lambda_*$ is the free energy of a soliton solution to the cubic NLS in $\R^3$.

In view of \eqref{Exhaust A}, we see that Theorem~\ref{T:main} follows from the following result.

\begin{thm}\label{main_c}
Let $u_0\in H^1_0(\Omega)$ satisfy \eqref{km}.  If
\begin{equation*}
F^\lambda(u_0)=E(u_0)+\lambda M(u_0)<F^\lambda_*
\end{equation*}
for some $0<\lambda<\infty$, then there exists a unique global solution $u$ to \eqref{nls} and
$$
\|u\|_{L_{t,x}^5(\R\times\Omega)}\leq C(\lambda, F^\lambda(u_0)).
$$
\end{thm}

As alluded above, we believe that \eqref{em} and \eqref{km} represent the sharp threshold in our setting, just as they do for the problem posed in $\R^3$.  More precisely, in \cite{hr:radial} it is shown that radial initial data $u_0\in H^1(\R^3)$ obeying \eqref{em} and
$$
\|\nabla u_0\|_{L^2(\R^3)}\|u_0\|_{L^2(\R^3)} > \|\nabla Q\|_{L^2(\R^3)}\|Q\|_{L^2(\R^3)}
$$
lead to solutions that blow up in finite time.  This is proved via the usual concave virial argument, which does not adapt directly to our setting --- the boundary term does not have a favorable sign.  It is natural to imagine that it should be possible to embed Euclidean blowup solutions into the exterior domain case via perturbative arguments; however, our current understanding of the structure of this blowup is not quite sufficient to push this through.  Nevertheless, we can show that our result is sharp in terms of uniform spacetime bounds:

\begin{prop}\label{P:blowup}
Fix $\lambda\in(0, \infty)$.  There exists a sequence of global solutions $u_n\in C_t H^1_0(\R\times\Omega)$ that satisfy \eqref{km} and
$$
F^\lambda(u_n)\nearrow F^\lambda_* \qtq{and} \|u_n\|_{L_{t,x}^5(\R\times\Omega)}\to \infty \quad\text{as}\quad n\to \infty.
$$
\end{prop}

This is proved in Section~\ref{S:blowup} by choosing a sequence $u_n$ that closely models $Q^\mu$ in shape, but is centered far from the obstacle.

\subsection*{Acknowledgements} R. K. was supported by NSF grant DMS-1265868. M. V. was supported by the Sloan Foundation and NSF grants DMS-0901166 and DMS-1161396. X. Z. was supported by the Simons Foundation.

%%%%%%%%%%%%%%%%%%%%%%%%%%%%%%%%%%%%%%%%%%%%%%%%%%%%%%%%%%%%%%%%%%%%%%%%%%%%%%%%
\section{Preliminaries}
%%%%%%%%%%%%%%%%%%%%%%%%%%%%%%%%%%%%%%%%%%%%%%%%%%%%%%%%%%%%%%%%%%%%%%%%%%%%%%%%

%%%%%%%%%%%%%%%%%%%%%%%%%%%%%%%%%%%%%%%%%%%%%%%%%%%%%%%%%%%%%%%%%%%%%%%%%%%%%%%%
\subsection{Notation and useful lemmas}
We write $X \lesssim Y$ or $Y \gtrsim X$ to indicate $X \leq CY$ for some constant $C>0$. We use the notation $X \sim Y$ whenever $X \lesssim Y \lesssim X$.

Throughout this paper, $\Omega$ denotes the exterior domain of a compact, smooth, strictly convex obstacle in $\R^3$. Without loss of generality, we assume that $0\in \Omega^c$ and
$\Omega^c\subset B(0, 1)$.  For any $x\in \R^3$ we use $d(x):=\dist(x,\Omega^c)$.

With $x_0\in \R^3$, we use $\tau_{x_0}$ to denote the translation operator $\tau_{x_0} f(x):=f(x-x_0)$.

Throughout this paper, $\chi$ will denote a smooth cutoff in $\R^3$ satisfying
\begin{align}\label{chi defn}
\chi(x)=
\begin{cases}
1, &\text{if }\ |x|\le \tfrac14\\
0,&\text{if }\ |x|>\tfrac12.
\end{cases}
\end{align}

We will use the following refined version of Fatou's lemma due to Brezis and Lieb.

\begin{lem}[Refined Fatou, \cite{BrezisLieb}]\label{lm:rf}
Let $0<p<\infty$ and assume that  $\{f_n\}\subseteq L^p(\R^d)$ with $\limsup_{n\to \infty}\|f_n\|_p<\infty$. If $f_n\to f$ almost everywhere, then
\begin{align*}
\int_{\R^d}\biggl||f_n|^p-|f_n-f|^p-|f|^p \biggr| \,dx\to 0 \qtq{as} n\to \infty.
\end{align*}
In particular, $\|f_n\|_p^p-\|f_n-f\|_p^p \to \|f\|_p^p. $
\end{lem}

We will use the following heat kernel estimate due to Q. S. Zhang.

\begin{lem}[Heat kernel estimate, \cite{Zhang:heat}]\label{lm:heat}
Let $\Omega$ denote the exterior of a smooth, compact, convex obstacle in $\R^d$ for $d\geq 3$.  Then there exists $c>0$ such that
\begin{align*}
|e^{t\ld}(x,y)|\lsm \Bigl(\tfrac{d(x)}{\sqrt t\wedge \diam}\wedge 1\Bigr)\Bigl(\tfrac{d(y)}{\sqrt t\wedge \diam}\wedge 1\Bigr) e^{-\frac{c|x-y|^2}t} t^{-\frac d 2},
\end{align*}
uniformly for $x, y\in \Omega$.  If either $x\notin \Omega$ or $y\notin \Omega$, then $e^{it\ld}(x,y)=0$.
\end{lem}

There is a natural family of Sobolev spaces associated to powers of the Dirichlet Laplacian.  Our notation for these is as follows:

\begin{defn}
For $s\geq 0$ and $1<p<\infty$, let $\dot H_D^{s,p}(\Omega)$ and $H_D^{s,p}(\Omega)$ denote the completions of $C^\infty_c(\Omega)$ under the norms
$$
\| f \|_{\dot H_D^{s,p}(\Omega)} := \| (-\Delta_\Omega)^{s/2} f \|_{L^p} \qtq{and} \| f \|_{H_D^{s,p}(\Omega)} := \| (1-\Delta_\Omega)^{s/2} f \|_{L^p}.
$$
When $p=2$ we write $\dot H^s_D(\Omega)$ and $H^s_D(\Omega)$ for $\dot H^{s,2}_D(\Omega)$ and $H^{s,2}_D(\Omega)$, respectively.
\end{defn}

When $\Omega$ is replaced by $\R^3$, these definitions lead to the classical $\dot H^{s,p}(\R^3)$ and $H^{s,p}(\R^3)$ families of spaces.  Note that the classical $H^{s,p}_0(\Omega)$ spaces are defined as subspaces of $H^{s,p}(\Omega)$, which are in turn defined as quotients of the spaces $H^{s,p}(\R^3)$.  Thus the spaces $H^{s,p}_0(\Omega)$ have no direct connection to the functional calculus of the Dirichlet Laplacian.  It is a well-known (but nontrivial) theorem that $H^{s,p}_0(\Omega) = H^{s,p}_D(\Omega)$ for $0<s<\smash[t]{\frac{1}{p}}$ and again for $\smash[t]{\frac{1}{p}}<s<1+\smash[t]{\frac{1}{p}}$.

The inter-relation of \emph{homogeneous} Sobolev spaces on non-compact manifolds throws up additional complications.  The case $s=1$ captures already the question of boundedness of Riesz transforms, which is a topic of on-going investigation.  Motivated by applications to NLS, in \cite{kvz:harmonic} we investigated when the two notions of Sobolev spaces  are equivalent in the case of exterior domains.  The advantage of such an equivalence is two fold:   Powers of the Dirichlet Laplacian commute with the linear evolution $e^{it\Delta_\Omega}$, while fractional product and chain rules (needed for treating the nonlinearity) have already been proved for spaces defined via powers of the Euclidean Laplacian.  Our findings in \cite{kvz:harmonic} are summarized in the following sharp result about the equivalence of Sobolev spaces.

\begin{thm}[Equivalence of Sobolev spaces, \cite{kvz:harmonic}]\label{T:equiv}
Let $d\geq 3$ and let $\Omega$ be the complement of a compact convex body $\Omega^c\subset\R^d$ with smooth boundary.  Let $1<p<\infty$.  If $0\leq s<\min\{1+\frac1p,\frac dp\}$ then
\begin{equation*}
\bigl\| (-\Delta_{\R^d})^{s/2} f \bigl\|_{L^p}  \sim_{d,p,s} \bigl\| (-\Delta_\Omega)^{s/2} f \bigr\|_{L^p}  \qtq{for all} f\in C^\infty_c(\Omega).
\end{equation*}
\end{thm}

Theorem~\ref{T:equiv} allows us to transfer directly several key results from the Euclidean setting to exterior domains.  One example is the $L^p$-Leibnitz (or product) rule for first derivatives:

\begin{cor}[Fractional product rule]\label{C:product rule}
For all $f, g\in C_c^{\infty}(\Omega)$, we have
\begin{align*}
\| (-\Delta_\Omega)^{\frac 12}(fg)\|_{L^p(\Omega)} \lesssim \| (-\Delta_\Omega)^{\frac 12} f\|_{L^{p_1}(\Omega)}\|g\|_{L^{p_2}(\Omega)}+
\|f\|_{L^{q_1}(\Omega)}\| (-\Delta_\Omega)^{\frac 12} g\|_{L^{q_2}(\Omega)}
\end{align*}
with the exponents satisfying $1<p, p_1, q_2<\infty$, $1<p_2,q_1\le \infty$, and $\frac1p=\frac1{p_1}+\frac1{p_2}=\frac1{q_1}+\frac1{q_2}$.
\end{cor}

The following simple lemma will be frequently used in this paper.

\begin{lem}\label{lm:often} Let $\phi\in H^1(\R^3)$, $\chi$ be as in \eqref{chi defn}, and $R_n\to \infty$. Then
\begin{align*}
&\bigl\|\chi\bigl(\tfrac{x}{R_n}\bigr)\phi\bigr\|_{H^1(\R^3)}\lsm \|\phi\|_{\hr},\\
&\lim_{n\to \infty}\bigl\|\bigl[1-\chi\bigl(\tfrac{x}{R_n}\bigr)\bigr]\phi\bigr\|_{\hr}=0.
\end{align*}
\end{lem}

\begin{proof} We only prove the second assertion.  By H\"older's inequality,
\begin{align*}
\bigl\|\bigl[1&-\chi\bigl(\tfrac{x}{R_n}\bigr)\bigr]\phi\bigr\|_{\hr}\\
&\lsm \bigl\|\bigl[1-\chi\bigl(\tfrac{x}{R_n}\bigr)\bigr]\phi\bigr\|_{L^2(\R^3)}+ \bigl\|\bigl[1-\chi\bigl(\tfrac{x}{R_n}\bigr)\bigr]\nabla \phi\bigr\|_{L^2(\R^3)}
	+R_n^{-1}\bigl\|\phi(\nabla\chi)\bigl(\tfrac{x}{R_n}\bigr)\bigr\|_{L^2(\R^3)}\\
&\lsm \|\phi\|_{L^2(|x|\gtrsim R_n)}+\|\nabla \phi\|_{L^2(|x|\gtrsim R_n)}+\|\nabla \chi\|_{L^3(\R^3)}\|\phi\|_{L^6{(|x|\sim R_n)}}.
 \end{align*}
 The claim now follows from the monotone convergence theorem.
 \end{proof}

By exploiting the functional calculus for self-adjoint operators, one can define the Littlewood--Paley projections adapted to $\Delta_\Omega$.  Just like their Euclidean counterparts, these operators obey Bernstein estimates.

\begin{lem}[Bernstein estimates]
Let $1<p<q\le \infty$ and $-\infty<s<\infty$. Then for any $f\in C_c^{\infty}(\Omega)$, we have
\begin{align*}
\|\po_{\le N} f \|_{\lpo}+\|\po_N f\|_{\lpo}&\lsm \|f\|_{\lpo},\\
\|\po_{\le N} f\|_{L^q(\Omega)}+\|\po_N f\|_{L^q(\Omega)}&\lsm N^{3(\frac 1p-\frac1q)}\|f\|_{L^p(\Omega)},\\
N^s\|\po_N f\|_{\lpo}&\sim \|(-\ld)^{\frac s2}\po_N f\|_{\lpo}.
\end{align*}
\end{lem}

%%%%%%%%%%%%%%%%%%%%%%%%%%%%%%%%%%%%%%%%%%%%%%%%%%%%%%%%%%%%%%%%%%%%%%%%%%%%%%%%%
\subsection{Strichartz estimates, local smoothing, and the virial identity}\leavevmode\\
Strichartz estimates for domains exterior to a compact, smooth, strictly convex obstacle were proved by Ivanovici \cite{iva:convex}; see also \cite{bss:schrodinger}.  Ivanovici obtained the full range of Strichartz estimates known in the Euclidean setting, with the exception of the endpoint $L_t^2 L_x^{6}$.  As we will only be using a finite collection of (non-endpoint) Strichartz norms in this paper, we can encapsulate everything into the following Strichartz spaces: For $\eps>0$ sufficiently small we define
\begin{align*}
S^0(I)&:=L_t^\infty L_x^2(I\times \Omega)\cap L_t^{2+\eps}L_x^{\frac{6(2+\eps)}{2+3\eps}}(I\times\Omega).
\end{align*}
Further, we define $N^0(I)$ as the corresponding dual Strichartz space and
$$
N^1(I):=\{f:\, f,(-\Delta_\Omega)^{\frac 12} f\in N^0(I)\}.
$$
Additionally, in our discussion of solutions in the whole Euclidean space, it will be convenient to use
\begin{align*}
S^1(I)&:=\{u: I\times\R^3\to \C:\, (1-\Delta_{\R^3})^{\frac 12} u \in L_t^\infty L_x^2(I\times \R^3)\cap L_t^{2}L_x^{6}(I\times \R^3)\}.
\end{align*}
Note that we will only use the $S^1(I)$ notation in the context of solutions in the whole Euclidean space.

With these notations, the Strichartz estimates read as follows:

\begin{thm}[Strichartz estimates, \cite{iva:convex}]\label{T:Strichartz}
Let $I\subset\R$ be a time interval and let $t_0\in I$. Then the solution $u:I\times\Omega\to \C$ to
\begin{align*}
iu_t+\Delta_\Omega u=f
\end{align*}
satisfies
\begin{align*}
\|u\|_{ S^0(I)}\le \|u(t_0)\|_{L^2(\Omega)}+\|f\|_{ N^0(I)}.
\end{align*}
In particular,
\begin{align*}
\|u\|_{L^{5}_t H^{1,\frac{30}{11}}_0(I\times\Omega)}\lsm \|u(t_0)\|_{H^1_0(\Omega)}+\|f\|_{N^1(I)}.
\end{align*}
\end{thm}

Using Theorems~\ref{T:Strichartz} and~\ref{T:equiv}, and arguing in the usual manner (cf. \cite{ckstt}), one obtains the following stability result for \eqref{nls}.

\begin{lem}[Stability]\label{lm:stability}
Let $\Omega$ be the exterior of a compact, smooth, strictly convex obstacle in $\R^3$. Let $I\subset\R$ be a time interval and let $\tilde u$ be an approximate solution to \eqref{nls} on $I\times\Omega$ in the sense that
\begin{align*}
i\partial_t \tilde u+\Delta_{\Omega} \tilde u=-|\tilde u|^2\tilde u+e
\end{align*}
for some function $e$.  Assume that
\begin{align*}
\|\tilde u\|_{L_t^{\infty} H^1_0(I\times\Omega)}\le \mathcal E \qtq{and} \|\tilde u\|_{L_{t,x}^{5}(I\times\Omega)}\le L
\end{align*}
for some positive constants $\mathcal E$ and $L$.  Let $t_0\in I$ and $u_0\in H^1_0(\Omega)$, and assume the smallness conditions
\begin{align*}
\|\tilde u(t_0)-u_0\|_{ H^1_0(\Omega)}\le\eps \qtq{and} \| e\|_{ N^1(I)}\le \eps
\end{align*}
for some $0<\eps<\eps_1=\eps_1(\mathcal E,L)$.  Then there exists a unique strong solution $u:I\times\Omega\to \C$ to \eqref{nls} with initial data $u_0$ at time $t=t_0$ satisfying
\begin{align*}
\|u-\tilde u\|_{L^{5}_x H^{1,\frac{30}{11}}_0(I\times\Omega)}&\leq C(\mathcal E,L)\eps.\end{align*}
\end{lem}

We will also use the local smoothing estimate.  The particular version we need is \cite[Lemma 2.13]{kvz:quintic}:

\begin{lem}[Local smoothing]\label{ls}
Let $u=e^{it\ld}u_0$. Then
\begin{align*}
\iint_{\ro}|\nabla u(t,x)|^2 \langle R^{-1}(x-z)\rangle^3 \,dx\,dt\lsm R\|u_0\|_{L^2(\Omega)}\|\nabla u_0\|_{L^2(\Omega)}
\end{align*}
uniformly for $z\in\R^3$ and $R>0$.
\end{lem}

A direct consequence of the local smoothing estimate is the following result, which will be used in the proof of the Palais--Smale condition.  For a similar statement adapted to the energy-critical problem in Euclidean space, see \cite[Lemma~2.5]{KV:AJM}.

\begin{cor}\label{cora}
Given $u_0\in H_0^1(\Omega)$, we have
\begin{align*}
\|\nabla e^{it\ld}u_0\|_{L_{t,x}^2(|t-\tau|\le T, |x-z|\le R)}\lsm R^{\frac{19}{30}}T^{\frac 1{10}}\|e^{it\ld} u_0\|_{\lt}^{\frac 13}
\|u_0\|_{H^1_0(\Omega)}^{\frac 23}.
\end{align*}
\end{cor}
\begin{proof} We split the left-hand side according to low and high frequencies.  To estimate the low frequencies, we use H\"older and Bernstein:
\begin{align*}
\|\nabla e^{it\ld}P_{\le N}^{\Omega} u_0\|_{L_{t,x}^2(|t-\tau|\le T, |x-z|\le R)}\lsm T^{\frac 3{10} } R^{\frac 9{10}}N\|e^{it\ld}u_0\|_{\lt}.
\end{align*}
To estimate the high frequencies, we use the local smoothing estimate Lemma~\ref{ls}:
\begin{align*}
\|\nabla e^{it\ld} P_{\ge N}^\Omega u_0\|_{L_{t,x}^2(|t-\tau|\le T,
|x-z|\le R)}\lsm R^{\frac 12}N^{-\frac 12}\|u_0\|_{H^1_0(\Omega)}.
\end{align*}
The claim follows by summing these two estimates and optimizing $N$.
\end{proof}

The last result in this subsection is a truncated virial inequality. Let $\phi$ be a smooth radial cutoff function such that
\begin{align*}
\phi(x):=\begin{cases}|x|^2, &\text{if } |x|\le 1\\
0,&\text{if } |x|\geq 2.
\end{cases}
\end{align*}
For $R\geq 1$, let $\phi_R(x):=R^2\phi\bigl(\frac xR\bigr)$.

\begin{lem}[Truncated Virial]\label{virial} Suppose $0\in\Omega^c$ and $R>100 \diam(\Omega^c)$. Then
\begin{align}
\partial_t \Im\int_\Omega \overline{ u(t,x)}\partial_k u(t,x)\partial_k\phi_R(x) \,dx&\ge \int_\Omega4|\nabla u(t,x)|^2-3|u(t,x)|^4 \,dx\notag\\
&\quad -O\Bigl(R^{-2}+\int_{|x|\ge R}|u(t,x)|^4+|\nabla u(t,x)|^2 \,dx\Bigr).\label{lv}
\end{align}
\end{lem}

\begin{proof}
We will exploit the local momentum conservation identity
\begin{equation}\label{moment}
\partial_t \Im(\partial_k u \,\bar u)=-2\partial_j \Re(\partial_k u\,\partial_j \bar u)+\tfrac 12 \partial_k \Delta (|u|^2) +\tfrac 12 \partial_k(|u|^4).
\end{equation}
Integrating this against $\partial_k\phi_R$, we obtain
\begin{align*}
\partial_t \Im\int_\Omega \bar u\partial_k u\partial_k\phi_R \,dx
&=-2\Re\int_\Omega \partial_j(\partial_k u\partial_j \bar u)\partial_k\phi_R \,dx\\
&\quad+\tfrac 12\int_\Omega \partial_k\Delta(|u|^2)\partial_k\phi_R \,dx +\tfrac 12\int_\Omega\partial_k(|u|^4)\partial_k \phi_R \,dx.
\end{align*}

We first seek lower bounds on each of the terms appearing on the right-hand side of the equality above.  From the divergence theorem and the fact that $\partial_{jk}\phi_R(x)=2\delta_{jk}$ for $|x|<R$, we obtain
\begin{align*}
&-2\Re\int_\Omega\partial_j(\partial_k u\,\partial_j \bar u)\,\partial_k \phi_R \,dx \\
&=-2\Re\int_\Omega\partial_j(\partial_k u\partial_j\bar u\partial_k\phi_R)\, dx+2\Re\int_\Omega\partial_k u\,\partial_j\bar u\,\partial_{jk}\phi_R \,dx\\
&=2\Re\int_{\partial \Omega}\nabla u\cdot \nabla \phi_R\nabla \bar u\cdot \vec n \,d\sigma(x) +2\Re\int_{\Omega}\partial_k u\,\partial_j\bar u\,\partial_{jk}\phi_R \,dx \\
%+2\Re\int_{|x|> R} \partial_k u\partial_j\bar u\partial_{jk}\phi_R \,dx\\
&\ge 2\int_{\partial\Omega}|u_n|^2(\phi_R)_n \,d\sigma(x)+4\int_\Omega |\nabla u|^2 \,dx-O\Bigl(\int_{|x|\ge R}|\nabla u|^2 \,dx\Bigr).
\end{align*}
Here $\vec n$ denotes the outer normal to $\partial \Omega$ (i.e., $\vec n$ points into $\Omega$), $u_n:=\nabla u\cdot \vec n$, and $(\phi_R)_n:= \nabla \phi_R\cdot \vec n$.  We have also used the fact that $\nabla u= u_n\ \vec n$, which follows from the Dirichlet boundary conditions.

Arguing similarly for the second term, we obtain
\begin{align*}
\tfrac 12\int_\Omega \partial_k\Delta(|u|^2)\partial_k\phi_R \,dx
&=-\tfrac 12\int_{\partial \Omega}\Delta(|u|^2)(\phi_R)_n \,d\sigma(x)-\tfrac12\int_\Omega \Delta(|u|^2) \Delta \phi_R \,dx\\
&=-\int_{\partial\Omega}|\nabla u|^2(\phi_R)_n \,d\sigma(x)-\tfrac 12\int_{\Omega}|u|^2\Delta \Delta \phi_R\,dx\\
&\ge -\int_{\partial \Omega}|\nabla u|^2(\phi_R)_n d\sigma(x)-O(R^{-2}).
\end{align*}

The third term can be estimated as follows:
\begin{align*}
\tfrac 12\int_\Omega \partial_k(|u|^4)\partial_k\phi_R \,dx&=-\tfrac 12\int_\Omega |u|^4 \Delta\phi_R \,dx\ge -3\int_\Omega |u|^4 \,dx-O\Bigl(\int_{|x|\geq R}|u|^4 \,dx\Bigr).
\end{align*}

Putting all the pieces together and noting that  $\nabla \phi_R(x)= 2x$ on $\partial \Omega$, we deduce
\begin{align*}
\partial_t \Im \int_\Omega\bar u\,\partial_k u\,\partial_k\phi_R \,dx
&\ge \int_\Omega4|\nabla u|^2-3|u|^4 dx+2\int_{\partial\Omega}|\nabla u|^2 x\cdot\vec n \,dx\\
&\quad -O\Bigl(R^{-2}+\int_{|x|\geq R}|u|^4 +|\nabla u|^2\,dx\Bigr).
\end{align*}
Finally, as $\partial \Omega$ is convex we have that $x\cdot \vec n\ge 0$, which immediately leads to \eqref{lv}.
\end{proof}

%%%%%%%%%%%%%%%%%%%%%%%%%%%%%%%%%%%%%%%%%%%%%%%%%%%%%%%%%%%%%%%%%%%%%%%%%%%%%%%%%
\subsection{Convergence results}
The defects of compactness in the Strichartz inequality
$$
\|e^{-it\Delta_\Omega}f\|_{L_{t,x}^5(\R\times\Omega)} \lesssim \| f\|_{H^1_0(\Omega)}
$$
are the same as in the Euclidean case, namely, spacetime translations.  (Scaling is not an issue because $L^5_{t,x}$ has dimensionality strictly between that of $L^2_x$ and $\smash{\dot H^1_x}$.)  In the Euclidean case, these defects of compactness are associated to exact symmetries of the equation.  In our case, however, the obstacle breaks the space translation symmetry.  Correspondingly, our linear profile decomposition must handle possible changes in geometry.  This issue was systematically studied in \cite{kvz:quintic} (where a scaling symmetry was also present).  In this paper, we record only the relevant convergence results from \cite{kvz:quintic}, namely, when the obstacle is marching away to infinity relative to the initial data.  This scenario gives rise to the whole Euclidean space $\R^3$ as the limiting geometry.

\begin{prop}[Convergence of domains, \cite{kvz:quintic}]\label{conv}
Suppose $\{x_n\}\subset\Omega$ are such that $|x_n|\to \infty$ and write $\Omega_n:=\Omega-\{x_n\}$.  For $h\in C_c^\infty(\R^3)$ and  $\Theta\in C_c^\infty(0,\infty)$ we have
\begin{align*}
&\lim_{n\to\infty}\|e^{it\Delta_{\Omega_n}}h-e^{it\Delta_{\R^3}}h\|_{\dot H^{-1}(\R^3) \cap \dot H^{1}(\R^3)}=0,\\
&\lim_{n\to\infty}\|[\Theta(-\Delta_{\Omega_n})-\Theta(-\Delta_{\R^3})]\delta(y)\|_{\dot H^{-1}(\R^3)}=0,\\
&\lim_{n\to\infty}\|(-\Delta_{\Omega_n})^{\frac 12}h-(-\Delta_{\R^3})^{\frac 12}h\|_{L^2(\R^3)}=0,\\
&\lim_{n\to \infty}\|e^{it\Delta_{\Omega_n}}h-e^{it\Delta_{\R^3}}h\|_{L_{t,x}^5(\R\times\R^3) \cap L^5_t L^{\frac{30}{11}}_x (\R\times\R^3)}=0.
\end{align*}
The second limit above is uniform in $y$ on compact subsets in $\R^3$.
\end{prop}

\begin{proof}
The first two assertions follow from Proposition~3.6 in \cite{kvz:quintic}, which asserts convergence in $\dot H^{-1}_x$.  This implies weak convergence in $\dot H^1_x$
which we can then upgrade to strong convergence since
$$
\| e^{it\Delta_{\R^3}}h \|_{\dot H^{1}(\R^3)} = \| h\|_{\dot H^{1}(\R^3)} = \|e^{it\Delta_{\Omega_n}}h\|_{\dot H^{1}(\R^3)}
$$
by energy conservation for the free propagator.

The third relation is Lemma~3.7 from \cite{kvz:quintic}.  The last equation follows directly from Theorem~4.1 of \cite{kvz:quintic}, for exponent pair $(5,\frac{30}{11})$, and from interpolation between it and Corollary~4.2 of that paper, for exponent pair $(5,5)$.
\end{proof}

Proposition~\ref{lim} below is needed to prove asymptotic decoupling of parameters in the linear profile decomposition. The two statements made by this proposition are essentially Lemmas~5.4 and 5.5 in \cite{kvz:quintic}.

\begin{prop}[Weak Convergence, \cite{kvz:quintic}] \label{lim}
Assume $\Omega_n=\Omega$ or $\Omega_n=\Omega-\{y_n\}$ with $|y_n|\to \infty$. Then the following two statements hold:

\smallskip
\noindent $1$.  Let $(t_n,x_n)\in \R\times\R^3 $ satisfy $|t_n|+|x_n|\to \infty$ as $n\to \infty$. Then for $f\in C_c^\infty( \Omega)$ if $\Omega_n=\Omega$, or for $f\in C_c^\infty(\R^3)$ if $\Omega_n=\Omega-\{y_n\}$,
\begin{align*}
\tau_{x_n}e^{it_n\Delta_{\Omega_n}}f\rightharpoonup 0 \mbox{ weakly in } H^1(\R^3).
\end{align*}

\noindent $2$. Let $f_n\in H_0^1(\Omega_n)$ be such that $f_n\rightharpoonup 0$
weakly in $H^1(\R^3)$. Let $t_n\to t_\infty\in\R$. Then
\begin{align*}
e^{it_n\Delta_{\Omega_n}}f_n\rightharpoonup 0 \mbox{ weakly in } H^1(\R^3).
\end{align*}
\end{prop}

The next lemma, the last for this subsection, will be used to prove decoupling of the $L^4$-norms in Proposition~\ref{invers}.

\begin{lem}[Weak dispersive estimate]\label{lm:wd}
Let $\phi\in C_c^\infty(\Omega)$ and $\psi\in C_c^\infty(\R^3)$.  Let $\Omega_n:=\Omega-\{y_n\}$ with $|y_n|\to \infty$. Then for any sequence $t_n\to \infty$,
\begin{align*}
\lim_{n\to\infty}\|e^{it_n\Delta_\Omega}\phi\|_{L^4_x}=0 \qtq{and} \lim_{n\to\infty}\|e^{it_n\Delta_{\Omega_n}}\psi\|_{L^4_x}=0.
\end{align*}
\end{lem}
\begin{proof} To prove that the first limit is zero, we consider the function
$$
F(t):=\int_\Omega |e^{it\Delta_\Omega}\phi(x)|^4 \,dx.
$$
From the Strichartz inequality, we know that $F(t)\in L_t^1(\R)$; indeed,
$$
\|e^{it\Delta_\Omega}\phi\|_{L_{t,x}^4(\R\times\Omega)}\lsm \|\phi\|_{\dot H^{1/4}_0(\Omega)}.
$$
On the other hand, as
\begin{align*}
\bigl|\tfrac d{dt} F(t)\bigr|
&\lsm \int |e^{it\Delta_\Omega}\phi|^3 |e^{it\Delta_\Omega}\Delta_\Omega \phi| \,dx
\lsm \|\Delta_\Omega\phi\|_{L^2_x}\|e^{it\Delta_\Omega}\phi\|_{L^6_x}^3\lsm \|\phi\|_{H_0^2(\Omega)}^4,
\end{align*}
we see that $F$ is uniformly continuous.  That the first limit is zero follows easily from these two facts.

For a Lipschitz function $f:\R\to [0, \infty)$ we have $f(t)^2\lesssim \int_\R f(s)\, ds$.  Combining this with the argument above and the fourth part of Proposition~\ref{conv}, we derive
$$
\lim_{n\to\infty} \sup_t \int_{\R^3} \bigl| [e^{it\Delta_{\Omega_n}}- e^{it\Delta_{\R^3}}]\psi(x)\bigr|^5 \,dx = 0.
$$
Combining this with the $L^{5/4}_x\to L^5_x$ dispersive estimate for the Euclidean propagator, we deduce that
$$
\int_\Omega |e^{it_n\Delta_{\Omega_n}}\psi(x)|^5 \,dx \to 0 \qtq{as} n\to\infty.
$$
The claim now follows from the conservation of mass and H\"older's inequality.
\end{proof}

%%%%%%%%%%%%%%%%%%%%%%%%%%%%%%%%%%%%%%%%%%%%%%%%%%%%%%%%%%%%%%%%%%%%%%%%%%%%%%%%%
\subsection{Coercivity of the energy}
The coercivity property, which is part of the variational characterization of the ground state, plays an important role throughout the proof.  The version we use in this paper is a minor adaptation of the one in \cite{hr:radial} and is informed by our needs when proving Theorem~\ref{main_c}.

\begin{prop}\label{coer} Fix $\lambda>0$  and let $u_0\in H^1_0(\Omega)$ satisfy
\begin{align*}
F^\lambda(u_0)<F^\lambda_*\qtq{and} \|\nabla u_0\|_{L^2(\Omega)}\|u_0\|_{L^2(\Omega)}<\|\nabla Q\|_{L^2(\R^3)}\|Q\|_{L^2(\R^3)}.
\end{align*}
Then the corresponding solution $u$ to \eqref{nls} is global.  Moreover, for all $t\in \R$ we have
\begin{align*}
\|\nabla u(t)\|_{L^2(\Omega)}\|u(t)\|_{L^2(\Omega)}<\|\nabla Q\|_{L^2(\R^3)}\|Q\|_{L^2(\R^3)}
\end{align*}
and
\begin{align*}
\tfrac 16\|\nabla u(t)\|_{L^2(\Omega)}^2\le E(u)\le \tfrac 12\|\nabla u(t)\|_{L^2(\Omega)}^2.
\end{align*}
In particular,
\begin{align*}
F^\lambda(u)\sim \|u(t)\|_{H_0^1}^2.
\end{align*}

Furthermore, with $\delta>0$ such that
\begin{align*}
F^\lambda(u_0)<(1-\delta)F^\lambda_*,
\end{align*}
there exists $c_\delta>0$ such that for all $t\in \R$,
\begin{align*}
&\|\nabla u(t)\|_{L^2(\Omega)}\|u(t)\|_{L^2(\Omega)}<(1-c_\delta)\|\nabla Q\|_{L^2(\R^3)}\|Q\|_{L^2(\R^3)},\\
&4\|\nabla u(t)\|_{L^2(\Omega)}^2-3\|u(t)\|_{L^4(\Omega)}^4\ge c_\delta \|\nabla u(t)\|_{L^2(\Omega)}^2.
\end{align*}
\end{prop}

%%%%%%%%%%%%%%%%%%%%%%%%%%%%%%%%%%%%%%%%%%%%%%%%%%%%%%%%%%%%%%%%%%%%%%%%%%%%%%%%%
\section{Linear profile decomposition}
%%%%%%%%%%%%%%%%%%%%%%%%%%%%%%%%%%%%%%%%%%%%%%%%%%%%%%%%%%%%%%%%%%%%%%%%%%%%%%%%%

A guiding heuristic of this paper is that parts of a solution living near the obstacle should be controlled by the virial identity localized to this region, while parts of the solution living far from the obstacle ought to be understood in terms of the problem with no obstacle (solved in \cite{dhr}).  The purpose of this section is to develop a linear profile decomposition, which ultimately, will allow us to partition a solution into such parts.  Indeed, the two cases appearing in the next result reflect precisely this dichotomy.

\begin{prop}[Inverse Strichartz inequality]\label{invers}
Assume $ \{f_n\}\subset H^1_0(\Omega)$ satisfies
\begin{align*}
\lim_{n\to \infty}\|f_n\|_{H^1_0(\Omega)}=A<\infty \qtq{and} \lim_{n\to \infty}\|e^{-it\Delta_\Omega}f_n\|_{L_{t,x}^5(\R\times\Omega)}=\eps>0.
\end{align*}
Then there exist a subsequence in $n$, $\{\phi_n\}\subset H_0^1(\Omega)$, $\{t_n\} \subset \R$, and $\{x_n\}\subset\Omega$ conforming to one of the two cases below, such that either $t_n\to \pm\infty$ or $t_n\equiv 0$ and
\begin{align}
\liminf_{n\to\infty}&\|\phi_n\|_{H^1_0(\Omega)}\gtrsim \tfrac {\eps^6}{A^5},\label{pi1}\\
\liminf_{n\to\infty}&\bigl\{\|f_n\|_{\dot H^s_0(\Omega)}^2-\|f_n-\phi_n\|_{\dot H_0^s(\Omega)}^2-\|\phi_n\|_{\dot H^s_0(\Omega)}^2\bigr\}=0 \quad\text{for } s=0, 1,\label{pi2}\\
\liminf_{n\to\infty}&\bigl\{\|f_n\|_{L^4(\Omega)}^4-\|f_n-\phi_n\|_{L^4(\Omega)}^4-\|\phi_n\|_{L^4(\Omega)}^4\bigr\}=0.\label{pi4}
\end{align}
The two cases are as follows:

\noindent \textbf{Case 1:} Along the subsequence, $x_n\to x_\infty\in \Omega$, there is a $\phi\in H^1_0(\Omega)$ so that
$e^{-it_n\Delta_\Omega}f_n\rightharpoonup\phi$ weakly in $H^1_0(\Omega)$, and  $\phi_n:=e^{it_n\Delta_\Omega}\phi$.

\noindent \textbf{Case 2:} Along the subsequence, $d(x_n)\to\infty$, there is a $\tilde \phi\in H^1(\R^3)$ so that
\begin{align*}
[e^{-it_n\Delta_\Omega}f_n](x+x_n)\rightharpoonup\tilde \phi(x) \qtq{weakly in} H^1(\R^3),
\end{align*}
and
\begin{align*}
\phi_n(x):=e^{it_n\Delta_\Omega}[(\chi_n\tilde \phi)(x-x_n)] \qtq{with} \chi_n(x):=\chi\bigl(\tfrac x{d(x_n)}\bigr).
\end{align*}
\end{prop}

\begin{proof} Let $\delta>0$ be a small number to be chosen later.  Using the Bernstein and Strichartz inequalities, we get
\begin{align*}
&\|e^{-it\Delta_\Omega}\po_{\le \delta\eps^2}f_n\|_{\lt}\lsm \delta^{\frac 12}\eps\|e^{-it\Delta_{\Omega}}f_n\|_{L_t^5L_x^{\frac{30}{11}}(\ro)}\lsm \delta^{\frac 12}\eps\|f_n\|_{L^2(\Omega)}\lsm \delta^{\frac 12}\eps A, \\
&\|e^{-it\Delta_\Omega}\po_{>(\delta\eps^2)^{-1}}f_n\|_{\lt}\lsm \delta^{\frac 12}\eps\|(-\Delta_\Omega)^{\frac 12}e^{-it\Delta_\Omega} f_n\|_{L_t^5L_x^{\frac{30}{11}}(\ro)}\lsm \delta^{\frac12}\eps A.
\end{align*}
Taking $\delta$ small enough so that $\delta^{\frac 12}A\ll 1$, we deduce that
\begin{align*}
\|e^{-it\Delta_\Omega} \po_{med} f_n\|_{\lt}\ge \tfrac {\eps}2,
\end{align*}
where $\po_{med}:=\po_{\delta\eps^2<\cdot\le(\delta\eps^2)^{-1}}$.  Using this, H\"older, and Strichartz, we obtain
\begin{align*}
\tfrac \eps 2\le \|e^{-it\Delta_\Omega}\po_{med}f_n\|_{L_{t,x}^{\frac{10}3}(\ro)}^{\frac 23}\|e^{-it\Delta_\Omega}\po_{med}f_n\|_{L_{t,x}^{\infty}(\ro)}^{\frac 13}
\lsm A^{\frac 23}\|e^{-it\Delta_\Omega}\po_{med}f_n\|_{L_{t,x}^{\infty}(\ro)}^{\frac 13}.
\end{align*}
Therefore, there exist $(t_n, x_n)\in \ro$ such that
\begin{align}\label{lbb}
\bigl |\prdnn \po_{med} f_n(x_n)\bigr|\gtrsim \tfrac{\eps^3}{A^2}.
\end{align}

Next we prove that
\begin{align}\label{lbxn}
\inf_n d(x_n)\gtrsim_{\eps, A}1.
\end{align}
Indeed, using Lemma~\ref{lm:heat} we see that
\begin{align}
\int_{\Omega} |e^{\ld}(x_n, y)|^2 \,dy&\lsm \int_\Omega\bigl |d(x_n)[d(x_n)+|x_n-y|]e^{-c|x_n-y|^2}\bigr|^2 \,dy\label{nn}\\
&\lsm d(x_n)^2[d(x_n)+1]^2. \notag
\end{align}
On the other hand, writing
\begin{align*}
\prdnn\po_{med}f_n(x_n)=\int_\Omega e^{\ld}(x_n,y)e^{-(it_n+1)\ld}\po_{med}f_n(y)\,dy
\end{align*}
and using \eqref{lbb}, \eqref{nn}, Cauchy--Schwarz, and the Mikhlin multiplier theorem, we get
\begin{align*}
\tfrac{\eps^3}{A^2}&\lsm d(x_n)[d(x_n)+1]\bigl\|e^{-(it_n+1)\ld}\po_{\le (\delta\eps^2)^{-1}}f_n\|_{L^2(\Omega)}\lesssim_{\eps, A} d(x_n)[d(x_n)+1].
\end{align*}
This leads directly to \eqref{lbxn}.

Passing to a subsequence, we may assume that $x_n$ converges to a point in $\Omega$ or marches off to infinity.  (Note that $x_n$ cannot converge to a point in $\partial\Omega$ by virtue of \eqref{lbxn}.) We use this criterion to distinguish the two cases:

\noindent\textbf{Case 1:} $d(x_n)\sim 1$ and $x_n\to x_\infty\in \Omega$,\\
\noindent\textbf{Case 2:} $d(x_n)\to \infty. $

In both cases we define
\begin{align*}
g_n(x):=[e^{-it_n\ld}f_n](x+x_n).
\end{align*}
Obviously, $g_n$ is supported on $\Omega_n:=\Omega-\{x_n\}$ and $\|g_n\|_{H^1_0(\on)}=\|f_n\|_{H^1_0(\Omega)}\le A$.  Passing to a further subsequence if necessary, we can choose $\tilde \phi\in H^1(\R^3)$ so that
\begin{align}\label{gnto}
g_n\rightharpoonup\tilde \phi\ \mbox{ weakly in } H^1(\R^3).
\end{align}

In Case 1, using that $x_n\to x_\infty$ and \eqref{gnto}, we obtain
\begin{align*}
\prdnn f_n\rightharpoonup\phi(x):=\tilde \phi(x-x_\infty) \ \mbox{weakly in } \hoo.
\end{align*}
That $\phi$ belongs to $\hoo$ follows from the fact that $\hoo$ is weakly closed in $H^1(\R^3)$.

Next we observe that by passing to a further subsequence if necessary, we may assume that $t_n\to t_0\in[-\infty, \infty]$.  The possibilities $t_0 = \pm\infty$ are permitted in the statement of the proposition; however, when $t_0\in\R$ we need to show that judicious changes allow us to take $t_n\equiv0$.  In Case 1, we may take $t_n\equiv 0$ by replacing $\phi$ by $e^{it_0\Delta_\Omega}\phi$ and invoking the strong convergence of the linear propagator.  In Case 2, replacing $\tilde\phi$ by $e^{it_0\Delta_{\R^3}}\tilde \phi$ we may also take $t_n\equiv 0$; indeed, the claim boils down to the assertion\begin{align}\label{318}
\lim_{n\to \infty}\bigl\|e^{it_n\Delta_\Omega}\tau_{x_n} [\chi_n\tilde\phi] - \tau_{x_n}[\chi_n e^{it_0\Delta_{\R^3}}\tilde\phi]\bigr\|_{H^1_0(\Omega)}=0.
\end{align}

By a density argument, it suffices to prove \eqref{318} with $\tilde \phi$ replaced by $\psi\in C_c^\infty(\R^3)$.  Note that for $n$ sufficiently large we have $\chi_n\psi=\psi$. Using this and a change of variables, we reduce \eqref{318} to
\begin{align}\label{321}
\lim_{n\to \infty}\bigl\|e^{it_n\Delta_{\Omega_n}}\psi - \chi_n e^{it_0\Delta_{\R^3}}\psi\bigr\|_{H^1_0(\Omega_n)}=0.
\end{align}
We prove \eqref{321} by breaking it into three pieces.  First, by taking the time derivative, we have
$$
 \bigl\|e^{it_n\Delta_{\Omega_n}} \psi - e^{it_0\Delta_{\Omega_n}}\psi \bigr\|_{H^1(\R^3)} \leq | t_n - t_0 | \| \Delta \psi \|_{H^1(\R^3)} \to 0 \qtq{as} n\to \infty.
$$
Secondly, by the first part of Proposition~\ref{conv},
$$
e^{it_0\Delta_{\Omega_n}}\psi \to e^{it_0\Delta_{\R^3}}\psi \quad \text{strongly in } H^1(\R^3) \text{ as }  n\to\infty.
$$
Thirdly, Lemma~\ref{lm:often} yields
\begin{align*}
\bigl\| (1 - \chi_n) e^{it_0\Delta_{\R^3}} \psi \bigr\|_{H^1(\R^3)} \to 0 \qtq{as} n\to\infty,
\end{align*}
which completes the proof of \eqref{321}.

It remains to prove \eqref{pi1} through \eqref{pi4}.  We discuss the two cases separately.

\smallskip

\noindent\textbf{Case 1:}  To prove \eqref{pi1} in this case amounts to showing that $\phi$ is nontrivial.  Let $h:=\po_{med}\delta(x_\infty)$. First we note that
 \begin{align}\label{h bounds}
 \|h\|_{L^2(\Omega)}\lsm (\delta\eps^2)^{-\frac 32} \qtq{and} \|h\|_{L^{\frac 54}(\Omega)}\lsm (\delta\eps^2)^{-\frac 35}.
 \end{align}
To continue, we write
\begin{align}
\langle \phi, h\rangle&=\lim_{n\to\infty}\langle \prdnn f_n, h\rangle\notag\\
&=\lim_{n\to\infty} \prdnn\po_{med}f_n(x_n)-\lim_{n\to\infty}\langle \prdnn f_n, \po_{med}[\delta(x_n)-\delta(x_\infty)]\rangle. \label{fih1}
\end{align}
The second limit vanishes. Indeed, basic elliptic theory shows that
\begin{align*}
\|\nabla v\|_{L^\infty(|x|\le R)}\lsm R^{-1}\|v\|_{L^\infty(|x|\le 2R)}+R\|\Delta v\|_{L^\infty(|x|\le 2R)}.
\end{align*}
We will apply this to $v(x):=[\po_{med}\prdnn f_n](x+x_n)$ with $R:=\frac 14 d(x_n)$.  Note that with this choice,
\begin{align*}
\|v\|_{L^\infty(|x|\leq 2R)}\lsm (\delta\eps^2)^{-\frac 32} A \qtq{and} \|\Delta v\|_{L^\infty(|x|\le 2R)}\lsm (\delta\eps^2)^{-\frac 72} A.
\end{align*}
We thus obtain
$$
\|\nabla [\prdnn \po_{med} f_n](x+x_n)\|_{L^{\infty}(|x|\le \frac{d(x_n)}4)}\lsm(\delta\eps^2)^{-\frac 72} A.
$$
Using this and the fundamental theorem of calculus, for $n$ large we get
\begin{align*}
\bigl |\langle \prdnn f_n, &\po_{med}[\delta(x_n)-\delta(x_\infty)]\rangle\bigr|\\
&\lsm |x_n-x_\infty|\|\nabla \prdnn \po_{med} f_n(x+x_n)\|_{L^{\infty}(|x|\le \frac{d(x_n)}4)}\\
&\lsm |x_n-x_\infty|(\delta\eps^2)^{-\frac 72} A,
\end{align*}
which converges to zero as $n\to \infty$.  Thus, by \eqref{lbb}, \eqref{h bounds}, and \eqref{fih1},
\begin{align*}
\tfrac{\eps^3}{A^2}\lsm|\langle \phi, h\rangle|\lsm\|h\|_{L^2(\Omega)} \|\phi\|_{L^2(\Omega)}\lsm  (\delta\eps^2)^{-\frac 32}\|\phi\|_{L^2(\Omega)},
\end{align*}
which completes the proof of \eqref{pi1}.

The decoupling in $\dot H^s_0(\Omega)$ for $s=0,1$ follows easily from the fact that these are Hilbert spaces.

It remains to prove \eqref{pi4}.  First we discuss the case when $t_n\to \pm \infty$.  By the triangle inequality and Lemma~\ref{lm:wd}, in this case we have
$$
\bigl| \|f_n\|_{L^4}-\|f_n-\phi_n\|_{L^4}\bigr| \leq \|\phi_n\|_{L^4} \to 0 \qtq{as} n\to \infty,
$$
which leads immediately to \eqref{pi4}.

We now turn to the case when $t_n\equiv0$.  By the construction of the linear profiles, we have that $f_n\rightharpoonup \phi$ weakly in $H^1(\R^3)$.  Thus, using Rellich--Kondrashov and passing to a subsequence, we obtain that $f_n\to \phi$ almost everywhere on $\R^3$.  The claim now follows from Lemma~\ref{lm:rf}.

\smallskip

\noindent\textbf{Case 2:} $d(x_n)\to\infty$ as $n\to \infty$.  Recall that in this case,
\begin{align*}
\phi_n(x)=\prdn[(\chi_n\tilde\phi)(x-x_n)] \qtq{with} \chi_n(x)=\chi\bigl(\tfrac{x}{d(x_n)}\bigr).
\end{align*}

We first prove the lower bound \eqref{pi1}.  By Lemma \ref{lm:often},
\begin{align*}
\|\phi_n\|_{\hoo}=\|\chi_n\tilde \phi\|_{H^1_0(\Omega_n)}\to \|\tilde\phi\|_{H^1(\R^3)}.
\end{align*}
Thus, \eqref{pi1} will follow from the following expression of the non-triviality of $\tilde \phi$:
\begin{align}\label{ltb}
\|\tilde \phi\|_{\hr}\gtrsim \tfrac{\eps^6}{A^2}.
\end{align}
To show \eqref{ltb}, we first note that $h:= P_{med}^{\R^3}\delta(0)$ satisfies the estimates in \eqref{h bounds}.  Next, we write
\begin{align*}
\langle \tilde\phi, h\rangle =\lim_{n\to \infty}\langle g_n, h\rangle
=\lim_{n\to\infty}\langle g_n, P_{med}^{\Omega_n}\delta(0)\rangle+\lim_{n\to \infty}\langle g_n, (P_{med}^{\R^3}-P_{med}^{\Omega_n})\delta(0)\rangle.
\end{align*}
The last limit vanishes by Proposition \ref{conv} and the uniform boundedness of $g_n$. Therefore, by \eqref{lbb},
\begin{align}\label{lb2}
|\langle \tilde\phi, h\rangle|\gtrsim \tfrac{\eps^3}{A^5},
\end{align}
from which \eqref{ltb} follows immediately by Cauchy--Schwarz.

To prove decoupling in $\dot H^s_0(\Omega)$ we write
\begin{align*}
&\|f_n\|_{\dot H^s_0(\Omega)}^2-\|f_n-\phi_n\|_{\dot H^s_0(\Omega)}^2-\|\phi_n\|_{\dot H^s_0(\Omega)}^2\\
&=2 \Re\langle f_n-\phi_n,\phi_n\rangle_{\dot H^s_0(\Omega)}\\
&=2 \Re\langle g_n-\chi_n\tilde \phi,\chi_n\tilde\phi\rangle_{\dot H^s_0(\Omega_n)}\\
&=2 \Re\langle g_n-\tilde\phi,\tilde \phi\rangle_{\dot H^s(\R^3)} + 2 \Re\langle (1-\chi_n)\tilde\phi,\chi_n\tilde \phi\rangle_{\dot H^s(\R^3)} - 2 \Re\langle g_n-\tilde\phi,(1-\chi_n)\tilde \phi\rangle_{\dot H^s(\R^3)}.
\end{align*}
Claim \eqref{pi2} follows from this by using \eqref{gnto}, the uniform boundedness of $g_n$ in $H^1_0(\Omega_n)$, and Lemma \ref{lm:often}.

We now turn to \eqref{pi4}.  If $t_n\to \pm \infty$, by the triangle inequality, Gagliardo--Nirenberg, and Lemmas~\ref{lm:often} and \ref{lm:wd}, we have
$$
\bigl| \|f_n\|_{L^4}-\|f_n-\phi_n\|_{L^4}\bigr| \leq \|\phi_n\|_{L^4}\lsm \|e^{it_n\Delta_{\Omega_n}} \tilde\phi\|_{L^4} + \| (1-\chi_n)\tilde\phi\|_{H^1} \to 0 \qtq{as} n\to \infty,
$$
which leads immediately to \eqref{pi4}.

Next we consider the case when $t_n\equiv0$.  From the construction of the linear profiles, we have $f_n(x+x_n)\rightharpoonup\tilde \phi(x)$ weakly in $H^1(\R^3)$.  Thus, using Rellich--Kondrashov and passing to a subsequence we deduce that $f_n(x+x_n)\to\tilde \phi(x)$ almost everywhere on $\R^3$.  Lemma~\ref{lm:rf} then gives
$$
\|f_n\|_{L^4}^4- \|f_n- \tau_{x_n}\tilde\phi\|_{L^4}^4 - \|\tilde\phi\|_{L^4}^4\to 0 \qtq{as} n\to \infty.
$$
Combining this with Lemma~\ref{lm:often} and Gagliardo--Nirenberg yields \eqref{pi4} in this case.

This completes the proof of Proposition~\ref{invers}.
\end{proof}

We are now ready to state the linear profile decomposition for bounded sequences in $H^1_0(\Omega)$.

\begin{thm}[A linear profile decomposition in $H^1_0(\Omega)$] \label{thm:lf} Let $\{f_n\}$ be a bounded sequence in $H^1_0(\Omega)$.  After passing to a subsequence, there exist $J^*\in\{0, 1, 2,\ldots,\infty\}$, $\{\phi_n^j\}_{j=1}^{J^*}\subset \hoo$, $\{t_n^j\}_{j=1}^{J^*}\subset \R$ such that for each $j$ either $t_n^j\equiv 0 $ or $t_n^j\to\pm\infty$, and $\{x_n^j\}_{j=1}^{J^*}\subset \Omega$ conforming to one of the following two cases for each~$j$:

\noindent\textbf{Case 1:} $x_n^j\to x_\infty^j\in \Omega$ and there exists $\phi^j\in \hoo$ so that $\phi_n^j:=e^{it_n^j\ld}\phi^j.$

\noindent\textbf{Case 2:} $d(x_n^j)\to \infty$ and there exists $\phi^j\in H^1(\R^3)$ so that
\begin{align*}
\phi_n^j&:=e^{it_n^j\ld}[(\chi_n^j\phi^j)(x-x_n^j)]=\tau_{x_n^j}[e^{it_n^j\Delta_{\Omega_n^j}}(\chi_n^j\phi^j)] \qtq{with}\chi_n^j(x)=\chi\bigl(\tfrac{x}{d(x_n^j)}\bigr).
\end{align*}

Moreover, for any finite $0\le J\le J^*$ we have the decomposition
\begin{align*}
f_n=\sum_{j=1}^J\phi_n^j+w_n^J
\end{align*}
with $w_n^J\in \hoo$ satisfying
\begin{align}
&\lim_{J \to J^*}\limsup_{n\to \infty}\|\propagateomega w_n^J\|_{\lt}=0,\notag\\
&\lim_{n\to\infty}\bigl\{M(f_n)-\sum_{j=1}^J M(\phi_n^j)-M(w_n^J)\bigr\}=0,\notag\\
&\lim_{n\to \infty}\bigl\{E(f_n)-\sum_{j=1}^J E(\phi_n^j)-E(w_n^J)\bigr\}=0,\notag\\
&[e^{-it_n^J\Delta_{\Omega}}w_n^J](x+x_n^J)\rightharpoonup 0 \qtq{weakly in} H^1(\R^3),\notag\\
&\lim_{n\to \infty}|x_n^j-x_n^k|+|t_n^j-t_n^k|=\infty \qtq{for each} j\neq k.\label{lp5}
\end{align}
\end{thm}

\begin{proof}
Proposition~\ref{invers} provides all the key estimates.  With this in place, one may just repeat the well-known argument from the Euclidean setting.  Specifically, one argues inductively, using Proposition~\ref{invers} to remove one bubble of concentration at a time.  Proposition~\ref{lim} is needed to prove the asymptotic decoupling statement \eqref{lp5}.  For the proof of a linear profile decomposition for bounded sequences in $\dot H^1_0(\Omega)$, where $\Omega$ is the complement of a compact, smooth, strictly convex obstacle, see \cite{kvz:quintic}.
\end{proof}

\section{Embedding of nonlinear profiles}

The next major milestone in the proof of Theorem~\ref{T:main} is to use the linear profile decomposition obtained in the previous section to derive a Palais--Smale condition for minimizing sequences of blowup solutions to \eqref{nls}.  This amounts to proving a nonlinear profile decomposition for solutions to \eqref{nls}, which combined with Lemma~\ref{lm:stability} yields the desired compactness for minimizing sequences of solutions.  In order to prove a nonlinear profile decomposition for solutions to \eqref{nls}, we have to address the possibility that the nonlinear profiles we will extract are solutions to the focusing cubic NLS in a different limiting geometry, namely, the Euclidean space $\R^3$.  In this section, we will see how to embed the nonlinear profiles which solve $\text{NLS}_{\R^3}$ back inside $\Omega$.  Specifically, we need to approximate these profiles \emph{globally in time} by solutions to \eqref{nls} that satisfy \emph{uniform} spacetime bounds.

Throughout this section we use the notation
$$
X^1(I):= L^5_tH^{1,\frac{30}{11}}(I\times\R^3),
$$
where $I\subseteq\R$ is an arbitrary time interval.  By Sobolev embedding, this norm controls $L^5_{t,x}$.  Note also that this is a space to which Theorem~\ref{T:equiv} applies.

 \begin{thm}\label{embed} Let $\{t_n\}\subset \R$ be such that $t_n\equiv 0$ or $t_n\to \pm\infty$. Let $\{x_n\}\subset \Omega$ be such that $d(x_n)\to \infty$.   Assume $\phi\in H^1(\R^3)$ satisfies
\begin{align}\label{cond}
\|\nabla \phi\|_{L^2(\R^3)}\|\phi\|_{L^2(\R^3)}<\|\nabla Q\|_{L^2(\R^3)}\|Q\|_{L^2(\R^3)}  \qtq{and} F^\lambda(\phi)<F^\lambda_*,
\end{align}
for some $0<\lambda<\infty$. Define
\begin{align*}
\phi_n:=e^{it_n\ld}[\tau_{x_n}(\chi_n\phi)].
\end{align*}
Then for $n$ sufficiently large, there exists a global solution $v_n$ to \eqref{nls} with initial data $v_n(0)=\phi_n$ which satisfies
\begin{align*}
\|v_n\|_{L_{t,x}^5(\R\times\Omega)}\lsm 1,
\end{align*}
with the implicit constant depending only on $\|\phi\|_{H^1(\R^3)}$.  Furthermore, for any $\eps>0$ there exists $N_\eps\in \N$ and $\psi_\eps\in C_c^\infty(\R\times\R^3)$ such that for all $n\ge N_\eps$,
\begin{align}\label{app}
\|v_n(t-t_n, x+x_n)-\psi_{\eps}(t,x)\|_{X^1(\R)}<\eps.
\end{align}
\end{thm}

\begin{proof}
We prove this theorem in five steps.

\textbf{Step 1:} Constructing global solutions to $\text{NLS}_{\R^3}$.

Let $\theta=1/100$. If $t_n\equiv 0$, we let $w_n$ and $w_\infty$ be solutions to $\text{NLS}_{\R^3}$ with initial data $w_n(0)=\phi_{\le d(x_n)^\theta}$ and $w_\infty(0)=\phi$, respectively. If $t_n\to \pm\infty$, we let $w_n$ and $w_\infty$ be solutions to $\text{NLS}_{\R^3}$ satisfying
\begin{align*}
\lim_{t\to\pm\infty}\|w_n(t)-e^{it\Delta_{\R^3}}\phi_{\leq d(x_n)^\theta}\|_{H^1_x}=0  \qtq{and} \lim_{t\to\pm\infty}\|w_\infty(t)-e^{it\Delta_{\R^3}}\phi\|_{H^1_x}=0.
\end{align*}
By \eqref{cond} and the global theory for $\text{NLS}_{\R^3}$ developed  in \cite{hr:radial}, we see that $w_\infty$ and $w_n$ for $n$ sufficiently large are global solutions.  Moreover, using also the results from \cite{dhr}, we have
\begin{align}\label{wn}
\begin{cases}
\ \|w_n\|_{S^1(\R\times\R^3)}+\|w_{\infty}\|_{S^1(\R\times\R^3)}\lsm 1, \\
\ \||\nabla|^s w_n\|_{S^1(\R\times\R^3)}\lsm d(x_n)^{s\theta} \qtq{for} s\ge 0,\\
\ \lim_{n\to \infty}\|w_n-w_\infty\|_{S^1(\R\times\R^3)}=0,
\end{cases}
\end{align}
where the first two estimates are uniform in $n$, for $n$ sufficiently large.

\textbf{Step 2:} Constructing the approximate solution to \eqref{nls}.

Fix $T>0$ to be chosen later and define
\begin{align}\label{E:vat defn}
\tilde v_n(t,x):=\begin{cases}
[\chi_nw_n](t,x-x_n), & |t|\le T, \\
e^{i(t-T)\ld}\tilde v_n(T, x), & t>T, \\
e^{i(t+T)\ld}\tilde v_n(-T, x), & t<-T.
\end{cases}
\end{align}
From the Strichartz inequality and \eqref{wn},
\begin{align}\label{tvn}
\|\tilde v_n\|_{\lt}\lsm \|\chi_n w_n\|_{L_{t,x}^5(\R\times\Omega_n)}+ \|\chi_nw_n(\pm T)\|_{H^1_0(\Omega_n)}\lsm 1.
\end{align}

\textbf{Step 3:}  Asymptotic agreement of the initial data:
\begin{align}\label{data}
\lim_{T\to\infty}\limsup_{n\to \infty}\|\tilde v_n(t_n)-\phi_n\|_{H^1_0(\Omega)}=0.
\end{align}

We first consider the case when $t_n\equiv 0$. Using a change of variables and Lemma~\ref{lm:often}, we have
\begin{align*}
\|\tilde v_n(0)-\phi_n\|_{H^1_0(\Omega)}= \|\chi_n \phi_{> d(x_n)^\theta}\|_{H^1_0(\Omega_n)}\to 0  \qtq{as} n\to \infty.
\end{align*}

Next we consider the case when $t_n\to \infty$; the case $t_n\to -\infty$ can be treated similarly.  In this case, we have $t_n>T$ for sufficiently large $n$.  Thus,
\begin{align*}
\|&\tilde v_n(t_n)-\phi_n\|_{H^1_0(\Omega)}\\
&=\|e^{-iT\Delta_{\Omega_n}}[\chi_nw_n(T)]-\chi_n\phi]\|_{H^1_0(\Omega_n)}\\
&\lsm \|\chi_n[w_n(T)-w_\infty(T)]\|_{H^1_0(\Omega_n)}+ \|(e^{-iT\Delta_{\Omega_n}}-e^{-iT\Delta_{\R^3}})[\chi_nw_\infty(T)]\|_{H^1(\R^3)}\\
&\quad+\|e^{-iT\Delta_{\R^3}}w_\infty(T)-\phi\|_{H^1(\R^3)}+\|(1-\chi_n)w_\infty(T)\|_{H^1(\R^3)}+\|(1-\chi_n)\phi\|_{H^1(\R^3)},
\end{align*}
which converges to zero by first taking $n\to \infty$ and then $T\to \infty$.  To see this, we employ \eqref{wn}, Proposition~\ref{conv} (after first approximating $w_\infty(T)$ by $h\in C^\infty_c(\R^3)$ and noting that $\chi_n h = h$ for $n$ sufficiently large), and Lemma~\ref{lm:often}.

\textbf{Step 4:} Proving $\tilde v_n$ is an approximate solution to \eqref{nls} in the sense that
\begin{align}\label{error}
\lim_{T\to \infty}\limsup_{n\to \infty}\|[(i\partial_t+\ld) \tilde v_n +|\tilde v_n|^2\tilde v_n\|_{N^1(\R)}=0.
\end{align}

Let $e:=(i\partial_t+\ld)\tilde v_n+|\tilde v_n|^2\tilde v_n$.  We first consider the contribution of $\{t>T\}$ to LHS\eqref{error}; the contribution of $\{t<-T\}$ can be treated similarly.  In this case, we have $e=|\tilde v_n|^2\tilde v_n$.  Thus, using Strichartz, Lemma~\ref{lm:often}, and \eqref{wn}, we estimate
\begin{align*}
\|e\|_{N^1((T, \infty))} &= \||\tilde v_n|^2 \tilde v_n\|_{L_t^{\frac53}H_0^{1, \frac{30}{23}}((T,\infty)\times\Omega)}\\
&\lsm \|\tilde v_n\|_{L_{t,x}^5((T,\infty)\times\Omega)}^2\|\tilde v_n\|_{L_t^5H_0^{1,\frac{30}{11}}((T, \infty)\times\Omega)}\\
&\lsm \|\tilde v_n\|_{L_{t,x}^5((T,\infty)\times\Omega)}^2\|\chi_n w_n(T)\|_{H^1_0(\Omega_n)}\\
&\lsm \|\tilde v_n\|_{L_{t,x}^5((T,\infty)\times\Omega)}^2.
\end{align*}
To continue, let $w_+$ denote the asymptotic state of $w_\infty$, that is,
\begin{equation}\label{E:scatstate}
\|w_\infty(t)-e^{it\Delta}w_+\|_{S^1([T,\infty)\times \R^3)}\to 0 \qtq{as} T\to \infty.
\end{equation}
That such a $w_+$ exists follows from the main theorem in \cite{dhr}.  By the triangle inequality and Strichartz,
\begin{align*}
\|\tilde v_n\|_{L_{t,x}^5((T,\infty)\times\Omega)}
&=\|e^{i(t-T)\Delta_{\Omega_n}}[\chi_n w_n(T)]\|_{L_{t,x}^5((T,\infty)\times\Omega_n)}\\
&\lsm \|(e^{i(t-T)\Delta_{\Omega_n}}-e^{i(t-T)\Delta_{\R^3}})[\chi_n w_\infty(T)]\|_{L_{t,x}^5((T,\infty)\times\Omega_n)}\\
&\quad+ \|e^{it\Delta_{\R^3}} w_+\|_{L_{t,x}^5((T,\infty)\times\R^3)}+\|e^{-iT\Delta_{\R^3}}w_\infty(T) -w_+\|_{H^1(\R^3)}\\
&\quad+\|(1-\chi_n)w_\infty(T)\|_{H^1(\R^3)}+\|\chi_n[w_n(T)-w_\infty(T)]\|_{H_0^1(\Omega_n)},
\end{align*}
which converges to zero by first taking $n\to \infty$ then $T\to \infty$, in view of Proposition~\ref{conv}, the monotone convergence theorem, Lemma~\ref{lm:often} and \eqref{wn}.  Thus the contribution of $\{t>T\}$ to LHS\eqref{error} is acceptable.

We are left to estimate the contribution of $\{|t|\le T\}$ to LHS\eqref{error}.  In this case, we compute
\begin{align*}
e=\tau_{x_n}\bigl[(\chi_n^3-\chi_n)|w_n|^2 w_n+2\nabla \chi_n\cdot\nabla w_n+(\Delta\chi_n) w_n\big].
\end{align*}
Using a change of variables, we estimate the contribution of these terms as follows:
\begin{align*}
&\|(\chi_n^3-\chi_n)|w_n|^2 w_n\|_{N^1([-T,T])}\\
&\lsm \|(\chi_n^3-\chi_n)|w_n|^2w_n\|_{L_{t,x}^{\frac{10}7}(\R\times\Omega_n)}+\|(3\chi_n^2-1)\nabla \chi_n |w_n|^2w_n\|_{L_{t,x}^{\frac{10}7}(\R\times\Omega_n)}\\
&\quad +\|(\chi_n^3-\chi_n)|w_n|^2\nabla w_n\|_{L_{t,x}^{\frac{10}7}(\R\times\Omega_n)}\\
&\lsm \|w_n\|_{L_{t,x}^5(\R\times\R^3)}^2\Bigl [\|w_n-w_\infty\|_{L_{t,x}^{\frac{10}3}(\R\times\R^3)}+\|w_\infty \|_{L_{t,x}^{\frac{10}3}(\R\times\{|x|\sim d(x_n)\})}\Bigr]\\
&\quad+\|\nabla \chi_n\|_{L_x^3}\|w_n\|_{L_t^4L_x^{12}(\R\times\R^3)}^2\Bigl[\|w_\infty\|_{L_{t,x}^5(\R\times\{|x|\sim d(x_n)\})}+\|w_\infty-w_n\|_{L_{t,x}^5(\R\times\R^3)}\Bigr]\\
&\quad+\|w_n\|_{L_t^{\frac{10}3}H^{1,\frac{10}3}_0(\R\times\Omega_n)}\bigl[\|w_\infty\|_{L_{t,x}^5(\R\times\{|x|\sim d(x_n)\})}+\|w_\infty-w_n\|_{L_{t,x}^5(\R\times\R^3)}\bigr]^2,\end{align*}
which converges to zero as $n\to\infty$ by \eqref{wn} and the monotone convergence theorem.  Lastly,
\begin{align*}
\|2&\nabla \chi_n \cdot \nabla w_n +\Delta\chi_n w_n\|_{N^1([-T,T])}\\
&\lsm \|\nabla\chi_n\cdot\nabla w_n\|_{L_t^1H_0^1(\R\times\Omega_n)}+\|\Delta \chi_nw_n\|_{L_t^1H^1_0(\R\times\Omega_n)}\\
&\lsm T\Bigl[\|\nabla \chi_n\|_{L^\infty_x}\|\nabla w_n\|_{L_t^\infty L_x^2}+\|\nabla \chi_n\|_{L^\infty_x}\|\Delta w_n\|_{L_t^\infty L_x^2}+\|\Delta \chi_n\|_{L^\infty_x}\|\nabla w_n\|_{L_t^\infty L_x^2}\\
&\qquad\qquad + \|\Delta \chi_n\|_{L^\infty_x}\|w_n\|_{L_t^\infty L_x^2}+\|\nabla\Delta \chi_n\|_{L^\infty_x}\|w_n\|_{L_t^\infty L_x^2}\Bigr]\\
&\lsm T d(x_n)^{\theta-1}\to 0\qtq{as} n\to \infty.
\end{align*}
Thus the contribution of $\{|t|\leq T\}$ to LHS\eqref{error} is also acceptable.

\textbf{Step 5:} Constructing $v_n$ and approximating by $C_c^\infty$ functions.

Using \eqref{tvn}, \eqref{data}, \eqref{error} and Lemma~\ref{lm:stability}, for $n$ sufficiently large we obtain a global solution $v_n$ to \eqref{nls} with initial data $v_n(0)=\phi_n$ which satisfies
\begin{align}\label{approx}
\|v_n\|_{\lt}\lsm 1\qtq{and} \lim_{T\to\infty}\limsup_{n\to\infty}\|v_n(t-t_n)-\tilde v_n(t)\|_{X^1(\R)}=0.
\end{align}

It remains to prove \eqref{app}.  From the density of $C_c^\infty(\R\times\R^3)$ functions in $X^1(\R)$, for any $\eps>0$ we can find $\psi_\eps\in C_c^\infty(\R\times\R^3)$ such that
\begin{align*}
\|w_\infty-\psi_\eps\|_{X^1(\R)}\le \tfrac\eps 3.
\end{align*}
Using also \eqref{approx}, we see that to prove \eqref{app} it suffices to show
\begin{align}\label{final}
\|\tilde v_n(t,x)-w_\infty(t,x-x_n)\|_{X^1(\R)}\le \tfrac \eps 3
\end{align}
for $n, T$ sufficiently large. By \eqref{E:vat defn} and the triangle inequality,
\begin{align*}
\text{LHS}\eqref{final}&\le \|\chi_n w_n-w_\infty\|_{X^1([-T,T])}+\|e^{i(t-T)\Delta_{\Omega_n}}[\chi_nw_n(T)]-w_\infty\|_{X^1((T,\infty))}\\
&\quad +\|e^{i(t+T)\Delta_{\Omega_n}}[\chi_nw_n(-T)]-w_\infty\|_{X^1((-\infty, -T))}.
\end{align*}
Arguing as in Lemma~\ref{lm:often} and using \eqref{wn}, we have
\begin{align*}
\|\chi_nw_n-w_\infty\|_{X^1([-T,T])}\le \|\chi_n(w_n-w_\infty)\|_{X^1(\R)}+\|(1-\chi_n)w_\infty\|_{X^1(\R)}\to 0
\end{align*}
as $n\to\infty$.  Next, we estimate the contribution of $\{t>T\}$; the contribution of $\{t<-T\}$ can be handled similarly.  To this end, we observe that
\begin{align*}
\|&e^{i(t-T)\Delta_{\Omega_n}}[\chi_n w_n(T)]-w_\infty\|_{X^1((T,\infty))}\\
&\lesssim \|w_n-w_\infty\|_{L^\infty_t H^1_x} + \| e^{i(t-T)\Delta_{\Omega_n}} \chi_n w_\infty(T)\|_{X^1((T,\infty))}
	+ \| w_\infty(t) \|_{X^1((T,\infty))}.
\end{align*}
The first term converges to zero as $n\to\infty$ by virtue of \eqref{wn}, while the last term converges to zero as $T\to\infty$ by the dominated convergence theorem.  This leaves us to estimate the middle term, which is bounded by
\begin{align*}
%\| e^{i(t-T)\Delta_{\Omega_n}} \chi_n w_\infty(T)\|_{X^1((T,\infty))} \lesssim {}&
	\| e^{i(t-T)\Delta_{\Omega_n}} \chi_n w_\infty(T)\|_{L^5_t L^{\frac{30}{11}}_x((T,\infty))} + \| e^{i(t-T)\Delta_{\Omega_n}} (-\Delta_{\Omega_n})^{\frac12}[\chi_n w_\infty(T)]\|_{L^5_t L^{\frac{30}{11}}_x((T,\infty))}.
\end{align*}
Using Proposition~\ref{conv} and \eqref{E:scatstate}, the first of these two summands is easily seen to converge to zero when first taking $n\to \infty$ and then $T\to \infty$.  Let us focus instead on the treatment of the second summand, which requires an additional idea:
\begin{align*}
\| &e^{i(t-T)\Delta_{\Omega_n}} (-\Delta_{\Omega_n})^{\frac12}[\chi_n w_\infty(T)]\|_{L^5_t L^{\frac{30}{11}}_x((T,\infty))} \\
& \lesssim \| (-\Delta_{\Omega_n})^{\frac12}[\chi_n w_\infty(T)] - \chi_n (-\Delta_{\R^3})^{\frac12} w_\infty(T) \|_{L^2_x} \\
& \qquad	+ \| [e^{i(t-T)\Delta_{\Omega_n}}-e^{i(t-T)\Delta_{\R^3}}]\chi_n (-\Delta_{\R^3})^{\frac12} w_\infty(T)\|_{L^5_t L^{\frac{30}{11}}_x((T,\infty))} \\
& \qquad	+ \| e^{i(t-T)\Delta_{\R^3}} \chi_n (-\Delta_{\R^3})^{\frac12} w_\infty(T)\|_{L^5_t L^{\frac{30}{11}}_x((T,\infty))}.
\end{align*}
Now observe that the first term converges to zero as $n\to\infty$ by virtue of part three of Proposition~\ref{conv} and Lemma~\ref{lm:often}; the second term converges to zero as $n\to\infty$ by virtue of part four of Proposition~\ref{conv}; the last term converges to zero as $n\to\infty$ and then $T\to\infty$ by virtue of \eqref{E:scatstate} and the dominated convergence theorem.  This completes the proof of Theorem~\ref{embed}.
\end{proof}

%%%%%%%%%%%%%%%%%%%%%%%%%%%%%%%%%%%%%%%%%%%%%%%%%%%%%%%%%%%%%%%%%%%%%%%%%%%%%%
\section{A Palais--Smale condition}
%%%%%%%%%%%%%%%%%%%%%%%%%%%%%%%%%%%%%%%%%%%%%%%%%%%%%%%%%%%%%%%%%%%%%%%%%%%%%%

The goal of this section is to prove a Palais--Smale condition for minimizing sequences of blowup solutions to \eqref{nls}.

Throughout this section, we use the notation
$$
S_I(u):=\int_I\int_{\Omega} |u(t,x)|^5\, dx\, dt,
$$
where $I\subseteq\R$ denotes a time interval.  For $\lambda\in (0, \infty)$, we define
\begin{align*}
L(F^\lambda):=\sup\{S_I(u)\},
\end{align*}
where the supremum is taken over all solutions $u:I\times\Omega\to \C$ to \eqref{nls} such that
$$
F^\lambda(u)\le F^\lambda \qtq{and} \|\nabla u(t)\|_{L^2(\Omega)}\|u(t)\|_{L^2(\Omega)}<\|\nabla Q\|_{L^2(\R^3)}\|Q\|_{L^2(\R^3)}
$$
for some $t\in I$.  By Theorem~\ref{T:lwp} and Proposition~\ref{coer}, we have $L(F^\lambda)<\infty$, provided $F^\lambda$ is sufficiently small.  Indeed,
\begin{align}\label{small}
\|u\|_{X^1(\R)}\lsm_\lambda F^{\lambda}(u_0)^{\frac 12} \qtq{whenever} F^{\lambda}(u_0)\lsm_\lambda \eta_0,
\end{align}
where $\eta_0$ is the small data threshold given by Theorem~\ref{T:lwp}.

Therefore, if Theorem \ref{main_c} fails to be true, then there exist a $\lambda_0\in(0,\infty)$ and a critical value $0<F^{\lambda_0}_{c}<F^{\lambda_0}_*$ such that
 \begin{align}\label{induct}
L(F^{\lambda_0})<\infty \mbox{ for  }F^{\lambda_0}<F^{\lambda_0}_{c} \qtq{and} L(F^{\lambda_0})=\infty\mbox{ for } F^{\lambda_0}> F^{\lambda_0}_{c}.
 \end{align}

%We will use the fact that by Proposition~\ref{coer}, there exists a $\gamma=\gamma(\lambda_0, F^{\lambda_0}_c)>0$ such that
%\begin{align}\label{uniform}
%\|\nabla u(t)\|_2\|u(t)\|_2<(1-\gamma)\|\nabla Q\|_2\|Q\|_2 \qtq{whenever} F^{\lambda_0}(u_0)<F^{\lambda_0}_c.
%\end{align}

\begin{prop}[Palais--Smale condition]\label{prop:psc}
Let $u_n:I_n\times\Omega\to \C$ be a sequence of solutions to \eqref{nls} and let $t_n\in I_n$ satisfy
\begin{align}\label{kmc}
\|\nabla u_n(t_n)\|_{L^2(\Omega)}\|u_n(t_n)\|_{L^2(\Omega)}<\|\nabla Q\|_{L^2(\R^3)}\|Q\|_{L^2(\R^3)}
\end{align}
and
\begin{align*}
&\lim_{n\to \infty} F^{\lambda_0}(u_n)=F_c^{\lambda_0}<F^{\lambda_0}_* \qtq{and} \lim_{n\to \infty} S_{\ge t_n}(u_n)=\lim_{n\to \infty}S_{\le t_n}(u_n)=\infty.
\end{align*}
Then $\{u_n(t_n)\}$ is precompact in $\hoo$.
\end{prop}

\begin{proof}
Using the time-translation symmetry of \eqref{nls}, we may take $t_n\equiv0$ and so
\begin{align}\label{sinf}
\lim_{n\to\infty} S_{\ge 0}(u_n)=\lim_{n\to \infty} S_{\le 0} (u_n)=\infty.
\end{align}

Applying Theorem~\ref{thm:lf} to the sequence $u_n(0)$ (which is bounded in $H^1_0(\Omega)$ by hypothesis) and passing to a subsequence if necessary, we obtain
the linear profile decomposition
\begin{align}\label{s0}
u_n(0)=\sum_{j=1}^J \phi_n^j+w_n^J.
\end{align}
In particular, from energy and mass decoupling, for any given $J$ we have
\begin{align}
F^{\lambda_0}(u_n)=\sum_{j=1}^J F^{\lambda_0}(\phi_n^j)+F^{\lambda_0}(w_n^J)+o(1) \qtq{as} n\to \infty.\label{s01}
\end{align}

To prove the proposition we need to show that $J^*=1$, $w_n^1\to 0$ in $H^1_0(\Omega)$, $t_n^1\equiv0$, and that the only profile $\phi_n^1$ conforms to Case 1.  All other possibilities will be shown to contradict \eqref{sinf}.  We distinguish two scenarios.

\textbf{Scenario I:} $\sup_j\limsup_{n\to \infty}F^{\lambda_0}(\phi_n^j)=F^{\lambda_0}_c$.

From the non-triviality of the profiles, we have $\liminf_{n\to \infty} F^{\lambda_0}_c(\phi_n^j)>0$ for every $1\le j\le J^*$; indeed, $\|\phi_n^j\|_{H^1_0(\Omega)}\to\|\phi^j\|_{H^1(\R^3)}$. Thus, \eqref{s01} implies that there is a single profile in the decomposition \eqref{s0}, that is $J^*=1$, and we can write
\begin{align}\label{s02}
u_n(0)=\phi_n+w_n \qtq{with} \|w_n\|_{H^1_0(\Omega)}\to 0.
\end{align}

If $\phi_n$ conforms to Case 2, then by Theorem~\ref{embed} there exists a global solution $v_n$ to \eqref{nls} with initial data $v_n(0)=\phi_n$, which satisfies finite spacetime bounds, uniform for $n$ large.  By Lemma~\ref{lm:stability}, these spacetime bounds extend to $u_n$ for $n$ sufficiently large, which contradicts \eqref{sinf}. Therefore $\phi_n$ must conform to Case 1 and we have the decomposition
\begin{align*}
u_n(0)=e^{it_n\Delta_\Omega}\phi+w_n \qtq{with} \|w_n\|_{H^1_0(\Omega)}\to 0,
\end{align*}
and $t_n\equiv0$ or $t_n\to \pm\infty$.  If $t_n\equiv 0$ we obtain the desired compactness.  Thus, we only need to preclude the case $t_n\to \pm \infty$.

Let us suppose $t_n\to \infty$; the case $t_n\to -\infty$ can be treated symmetrically.  In this case, using Strichartz and the monotone convergence theorem we get
\begin{align*}
S_{\ge 0}(e^{it\ld}u_n(0))\lsm S_{\ge t_n}(e^{it\ld}\phi)+S_{\geq 0}(e^{it\ld}w_n)\to 0 \qtq{as} n\to \infty.
\end{align*}
By Lemma~\ref{lm:stability}, this implies that $S_{\geq 0}(u_n)\to 0$, which again contradicts \eqref{sinf}.

\textbf{Scenario II:} $\sup_j\limsup_{n\to \infty}F^{\lambda_0}(\phi_n^j)\le F^{\lambda_0}_c-2\delta$ for some $\delta>0$.

We first observe that for each finite $J\le J^*$ we have $F(\phi_n^j)\le F_c^{\lambda_0}-\delta$ for all $1\le j\le J$ and $n$ sufficiently large. Moreover, from mass and kinetic energy decoupling, Proposition~\ref{coer}, and \eqref{kmc}, there exists $\gamma=\gamma(\lambda_0, F^{\lambda_0}_c)>0$ such that
\begin{align}\label{ss01}
\|\nabla \phi_n^j\|_{L^2(\Omega)}\|\phi_n^j\|_{L^2(\Omega)}<\bigl(1-\gamma \bigr)\|\nabla Q\|_{L^2(\R^3)}\|Q\|_{L^2(\R^3)}
\end{align}
for all $1\le j\le J$ and $n$ sufficiently large.

If $j$ conforms to Case 1 and $t_n^j\equiv0$, we define $v^j:I^j\times\Omega\to \C$ to be the maximal-lifespan solution to \eqref{nls} with initial data $v^j(0)=\phi^j$. If instead $t_n^j\to \pm\infty$, we define $v^j:I^j\times\Omega\to \C$ to be the maximal-lifespan solution to \eqref{nls} which scatters to $e^{it\ld}\phi^j$ as $t\to \pm \infty$. In both cases, we define $v_n^j(t,x)=v^j(t+t_n^j,x)$.  Note that $v_n^j$ is also a maximal-lifespan solution on $I_n^j=I^j-\{t_n^j\}$.  Moreover, for $n$ sufficiently large we have $0\in I_n^j$ and
\begin{align}\label{vnj(0)}
\lim_{n\to \infty}\|v_n^j(0)-\phi_n^j\|_{H^1_0(\Omega)}= 0 .
\end{align}
Combining this with $F(\phi_n^j)\le F^{\lambda_0}_c-\delta$, \eqref{induct}, and \eqref{ss01}, we see that $v_n^j$ is a global solution to \eqref{nls} with finite spacetime bounds. Therefore, we can approximate $v_n^j$ by $C_c^\infty(\R\times\R^3)$ functions.  For any $\eps>0$ we choose $\tilde\psi^j_\eps\in C_c^\infty(\R\times\R^3)$ such that
\begin{align*}
\|v^j-\tilde\psi^j_\eps\|_{X^1(\R)}\le \eps/2.
\end{align*}
Define $\psi^j_\eps=\tau_{-x_\infty^j}\tilde\psi^j_\eps$.  Changing variables, for $n$ sufficiently large we get
\begin{align}\label{close}
\|v_n^j-\tau_{x_n^j}\psi^j_\eps(\cdot+t_n^j)\|_{X^1(\R)}<\eps.
\end{align}

If $j$ conforms to Case 2, we apply Theorem~\ref{embed} to obtain a global solution $v_n^j$ to \eqref{nls} with $v_n^j(0)=\phi_n^j$.  By Theorem~\ref{embed}, this solution has finite spacetime bounds and satisfies \eqref{close}.

In all cases, we may use \eqref{small} together with the bounds on the spacetime norms of $v_n^j$ to deduce
\begin{align}\label{vnj bounds}
\|v_n^j\|_{X^1(\R)}\lsm_{F^{\lambda_0}_c,\delta} F(v_n^j)^{\frac 12}\lsm_{F^{\lambda_0}_c,\delta} 1.
\end{align}
Combining this with \eqref{s01} yields
\begin{align}\label{big}
\limsup_{n\to\infty}\sum_{j=1}^J\|v_n^j\|_{X^1(\R)}^2\lsm_{F^{\lambda_0}_c,\delta}\sum_{j=1}^J F(\phi_n^j)\lsm_{F^{\lambda_0}_c,\delta} 1,
\end{align}
uniformly for finite $J\le J^*$.  Using also \eqref{close} and the asymptotic decoupling of parameters \eqref{lp5}, one immediately deduces that
\begin{align}\label{sum}
\limsup_{n\to \infty}\biggl\|\sum_{j=1}^J v_n^j\biggr\|_{X^1(\R)}\lsm_{F_c^{\lambda_0},\delta} 1 \quad \text{uniformly for finite } J\le J^*.
\end{align}
Moreover, the same argument combined also with \eqref{s01} shows that for given $\eta>0$ there exists $J'=J'(\eta)$ such that
\begin{align}\label{er}
\limsup_{n\to \infty}\biggl\|\sum_{j=J'}^J v_n^j\biggr\|_{X^1(\R)}\le \eta \quad\text{uniformly in } J\ge J'.
\end{align}

We are now ready to construct an approximate solution to \eqref{nls}. For each $n$ and $J$ we define
\begin{align}\label{unJ}
u_n^J:=\sum_{j=1}^J v_n^j+e^{it\Delta_\Omega} w_n^J.
\end{align}
Obviously $u_n^J$ is defined globally in time.  We will verify that for $n$ and $J$ large, $u_n^J$ is an approximate solution to \eqref{nls} with global finite spacetime bounds and that $u_n^J(0)$ is a good approximation for $u_n(0)$ in $H^1_0(\Omega)$.  An application of Lemma~\ref{lm:stability} then yields that $u_n$ satisfies finite spacetime bounds globally in time, uniformly in $n$ large.  This contradicts \eqref{sinf} and so Scenario 2 cannot occur.

We now turn to verifying the claims we made above about $u_n^J$.

\noindent\textbf{Claim 1:} Asymptotic matching of the initial data:  For any finite $J$ we have
$$
\lim_{n\to \infty}\|u_n^J(0)-u_n(0)\|_{H^1_0(\Omega)}=0.
$$
Indeed, this follows immediately from \eqref{s0}, \eqref{vnj(0)}, and \eqref{unJ}.

\noindent\textbf{Claim 2:} Finite spacetime bounds for $u_n^J$:
$$
\limsup_{n\to \infty}\|u_n^J\|_{X^1(\R)}\lsm_{F^{\lambda_0}_c,\delta}1 \quad\text{uniformly in } J.
$$
This follows easily from \eqref{sum} and the Strichartz estimate.

\noindent\textbf{Claim 3:} $u_n^J$ is an approximate solution to \eqref{nls} for $n,J$ large:
$$
\lim_{J\to J^*}\limsup_{n\to \infty} \|(i\partial_t+\ld)u_n^J + |u_n^J|^2 u_n^J\|_{N^1(\R)}=0.
$$
Denote $F(z):=-|z|^2 z$ and write
\begin{align}
(i\partial_t+\ld)u_n^J-F(u_n^J)&=\sum_{j=1}^J F(v_n^j)-F(u_n^J)\notag\\
&=\sum_{j=1}^J F(v_n^j)-F\biggl(\sum_{j=1}^J v_n^j\biggr)+F(u_n^J-e^{it\ld}w_n^J) -F(u_n^J).\label{pt}
\end{align}
Using the equivalence of Sobolev spaces, we obtain
\begin{align*}
\biggl\|\sum_{j=1}^J &F(v_n^j)-F\biggl(\sum_{j=1}^J v_n^j\biggr)\biggr\|_{N^1(\R)}\\
&\lsm _{J} \sum_{j\neq k}\|v_n^k|v_n^j|^2\|_{L_t^{\frac{5}{3}} L_x^{\frac{30}{23}}}+\sum_{j\neq k}\|v_n^j v_n^k \nabla v_n^j\|_{L_t^{\frac{5}{3}} L_x^{\frac{30}{23}}}
	+\sum_{j\neq k} \||v_n^j|^2 \nabla v_n^k\|_{L_t^{\frac{5}{3}} L_x^{\frac{30}{23}}}\\
&\lsm_J\sum_{j\neq k}\|v_n^j v_n^k\|_{L_{t,x}^{\frac 52}} \|v_n^j \|_{X^1} +\sum_{j\neq k}\|v_n^j\|_{L_{t,x}^5}\|v_n^j \nabla v_n^k\|_{L_t^{\frac{5}{2}} L_x^{\frac{30}{17}}},
\end{align*}
which converges to zero as $n\to \infty$ in view of \eqref{vnj bounds} and the asymptotic orthogonality of parameters \eqref{lp5} combined with \eqref{close}.

We now turn to estimating the second difference in \eqref{pt}. We will show
\begin{align}\label{two}
\lim_{J\to \infty}\limsup_{n\to \infty}\|F(u_n^J-e^{it\ld}w_n^J)-F(u_n^J)\|_{N^1(\R)}=0.
\end{align}
Using H\"older, Claim 2, and Strichartz, we estimate
\begin{align*}
\|F(u_n^J)-F(u_n^J-e^{it\ld}w_n^J)\|_{L_t^{\frac{5}{3}} L_x^{\frac{30}{23}}}
&\lsm\bigl[\|u_n^J\|_{L_t^5 L_x^{\frac{60}{17}}}^2+\|e^{it\ld}w_n^J\|_{L_t^5 L_x^{\frac{60}{17}}}^2\bigr]\|e^{it\ld}w_n^J\|_{L_{t,x}^5}\\
&\lsm_{F^{\lambda_0}_c,\delta}\|e^{it\ld}w_n^J\|_{L_{t,x}^5},
\end{align*}
which converges to zero as $n\to \infty$ and $J\to \infty$.  Similarly,
\begin{align*}
\|\nabla[F(u_n^J)-F(u_n^J-&e^{it\ld}w_n^J)]\|_{N^0}\\
&\lsm\sum_{k=0}^1 \|\nabla u_n^J\|_{L^5_t L^{\frac{30}{11}}_x}\|u_n^J\|_{L_{t,x}^5}^k \|e^{it\ld}w_n^J\|_{L_{t,x}^5}^{2-k}\\
&\quad+\sum_{k=0}^2 \|u_n^J\|_{L_{t,x}^5}^k \|u_n^J\nabla e^{it\ld} w_n^J \|_{L_{t,x}^2}\|e^{it\ld} w_n^J\|_{L_{t,x}^5}^{2-k}\\
&\lsm_{F^{\lambda_0}_c,\delta}\sum_{k=0}^1\|e^{it\ld} w_n^J\|_{L_{t,x}^5}^{2-k}+ \|u_n^J\nabla e^{it\ld} w_n^J \|_{L_{t,x}^2}.
\end{align*}
To prove \eqref{two}, it thus remains to show that
\begin{align*}
\lim_{J\to \infty}\limsup_{n\to \infty}\|u_n^J\nabla e^{it\ld}w_n^J \|_{L_{t,x}^2}=0.
\end{align*}
Using \eqref{er}, H\"older, and Strichartz, this further reduces to showing that
\begin{align*}
\lim_{J\to \infty}\limsup_{n\to \infty}\|v_n^j\nabla e^{it\ld}w_n^J\|_{L_{t,x}^2}=0 \qtq{for each} 1\le j\le J'.
\end{align*}
This however follows easily from \eqref{close} and Corollary~\ref{cora}.

This completes the proof of Claim 3 and so the proof of the proposition.
\end{proof}

%%%%%%%%%%%%%%%%%%%%%%%%%%%%%%%%%%%%%%%%%%%%%%%%%%%%%%%%%%%%%%%%%%%%
\section{Proof of Theorem \ref{main_c}}
%%%%%%%%%%%%%%%%%%%%%%%%%%%%%%%%%%%%%%%%%%%%%%%%%%%%%%%%%%%%%%%%%%%%
We will prove Theorem~\ref{main_c} in two steps.  First we will show that the Palais--Smale condition guarantees that the failure of Theorem~\ref{main_c} implies the existence of almost periodic counterexamples.  In the second step, we will use a virial argument to preclude the existence of such almost periodic solutions.

\begin{thm}[Existence of almost periodic solutions]\label{ap}
Suppose Theorem~\ref{main_c} fails.  Then there exist $\lambda_0\in(0,\infty)$, a critical value $F_c^{\lambda_0}<F^{\lambda_0}_*$, and a global solution $u$ to \eqref{nls} with $F(u)=F^{\lambda_0}_c$, which blows up in both time directions in the sense that
\begin{align*}
S_{\ge 0} (u)=S_{\le 0}( u)=\infty,
\end{align*}
and whose orbit $\{u(t): \, t\in \R\}$ is precompact in $H^1_0(\Omega)$.
\end{thm}

\begin{proof}
If Theorem \ref{main_c} fails, then there exist $\lambda_0\in(0,\infty)$, a critical value $F_c^{\lambda_0}<F^{\lambda_0}_*$, and a sequence of solutions
$u_n:I_n\times\Omega\to \C$ such that $F(u_n)\to F_c^{\lambda_0}$ and $S_{I_n}(u_n)\to \infty.$ Let $t_n\in I_n$ be such that
$S_{\ge t_n}(u_n)=S_{\le t_n}(u_n)=\frac 12 S_{I_n}(u_n)$; then
\begin{align}\label{bu}
\lim_{n\to\infty} S_{\ge t_n}(u_n)=\lim_{n\to \infty}S_{\le t_n}(u_n)=\infty.
\end{align}
Applying Proposition~\ref{prop:psc} and passing to a subsequence if necessary, we find $\phi\in H^1_0(\Omega)$ such that $u_n(t_n)\to \phi$ in $H^1_0(\Omega)$. In particular, $F^{\lambda_0}(\phi)=F_c^{\lambda_0}$.

Let $u:I\times\Omega\to \C$ be the maximal-lifespan solution to \eqref{nls} with initial data $u(0)=\phi$. Using Lemma~\ref{lm:stability} and \eqref{bu}, we get
\begin{align}\label{62}
S_{\ge 0}(u)=S_{\le 0}(u)=\infty.
\end{align}

Note that as $F_c^{\lambda_0}<F^{\lambda_0}_*$, Proposition~\ref{coer} guarantees that $\|u\|_{L_t^\infty H^1_0(I\times\Omega)}\leq C(F_c^{\lambda_0})<\infty$.  Thus, by the standard local theory for \eqref{nls} (see Theorem~\ref{T:lwp}), we obtain that $u$ is global in time.

Finally, we note that the orbit of $u$ is precompact in $H^1_0(\Omega)$.  Indeed, for any sequence $\{t_n'\}\subset I$, \eqref{62} implies that $S_{\le t_n'}(u)=S_{\ge t_n'}(u)=\infty$.  An application of Proposition~\ref{prop:psc} then yields that $\{u(t_n')\}$ admits a subsequence that converges strongly in $H_0^1(\Omega)$.
\end{proof}

\begin{proof}[Proof of Theorem \ref{main_c}]
Assume that Theorem \ref{main_c} fails and let $u$ be the minimal blowup solution given by Theorem~\ref{ap}.  By Lemma \ref{virial},
\begin{align*}
\partial_t \Im\int_\Omega \bar u(t,x)\partial_k u(T,x)\partial_k\phi_R(x) \,dx &\geq\int_\Omega4|\nabla u(t,x)|^2-3|u(t,x)|^4 \,dx\\
&\quad -O\Bigl(R^{-2}+\int_{|x|\ge R}|u(t,x)|^4+|\nabla u(t,x)|^2 \,dx\Bigr).
\end{align*}
By Proposition~\ref{coer}, there exists a small constant $c>0$ such that
\begin{align*}
\int_\Omega4|\nabla u(t,x)|^2-3|u(t,x)|^4 \,dx\ge c,
\end{align*}
uniformly for $t\in\R$. Therefore, as the orbit of $u$ is precompact in $H^1_0(\Omega)$, we can guarantee that
\begin{align*}
\partial_t \Im\int_\Omega \bar u(t,x)\partial_k u(T,x)\partial_k\phi_R(x) \,dx\ge \tfrac c2,
\end{align*}
by taking $R$ large enough.  Integrating this over $[0,T]$ and using Cauchy--Schwarz, we get
\begin{align*}
\tfrac c2 T&\leq \Bigl|\Im\int_{\Omega}\bar u(T,x)\partial_k u(T,x)\partial_k\phi_R(x)\,dx-\Im\int_\Omega\bar u_0(x)\partial_k u_0(x)\phi_R(x) \,dx\Bigr|\\
&\lsm R\|u\|_{L_t^\infty L^2_x(\R\times\Omega)}\|\nabla u \|_{L_t^\infty L^2_x(\R\times\Omega)}\lsm R.
\end{align*}
Taking $T$ sufficiently large, we derive a contradiction.
\end{proof}

%%%%%%%%%%%%%%%%%%%%%%%%%%%%%%%%%%%%%%%%%%%%%%%%%%%%%%%%%%%%%%%%%%%%
\section{Proof of Proposition~\ref{P:blowup}}\label{S:blowup}
%%%%%%%%%%%%%%%%%%%%%%%%%%%%%%%%%%%%%%%%%%%%%%%%%%%%%%%%%%%%%%%%%%%%

The goal of this section is to prove Proposition~\ref{P:blowup}.  Recall that without loss of generality, we may assume that $0\in \Omega^c\subset B(0,1)$.

We will prove the proposition for $\lambda=1/2$, because for this value of $\lambda$, the optimizer is the unrescaled ground state $Q$:
\begin{align*}
\inf_{\mu>0} F^{1/2} (Q^\mu) = F^{1/2} (Q)=F^{1/2}_*.
\end{align*}
The proof for arbitrary $\lambda\in (0, \infty)$ runs exactly parallel, but with messier formulas.  Indeed, the general case can be reduced to $\lambda=1/2$ via rescaling.

To see that $Q$ does indeed correspond to $\lambda=1/2$, one needs to exploit the Pohozaev identities obeyed by solutions to $\Delta Q + Q^3 = Q$, namely,
\begin{equation}\label{E:Poho}
\int_{\R^3}|Q(x)|^2\, dx=\tfrac14\int_{\R^3}|Q(x)|^4\, dx = \tfrac13 \int_{\R^3} |\nabla Q(x)|^2\, dx.
\end{equation}

To construct our sequence of solutions with diverging spacetime bounds, we consider the initial data:
$$
\phi_{n}(x):=(1-\eps_n) \chi(x)Q(x+x_n) \qtq{where} \eps_n\to 0, \ |x_n|\to \infty,
$$
and $\chi$ is a smooth cutoff such that $\chi(x)=0$ for $|x|\leq 1$ and $\chi(x)=1$ for $|x|\geq 2$.  More precisely, for a fixed sequence $\{\eps_n\}$ converging to zero, we choose $\{x_n\}\subset \R^3$ marching off to infinity rapidly enough that
$$
\|\nabla \phi_{n}\|_{L^2(\Omega)} \|\phi_{n}\|_{L^2(\Omega)}< \|\nabla Q\|_{L^2(\R^3)}\|Q\|_{L^2(\R^3)} \qtq{and} F^{1/2}(\phi_n) \nearrow F^{1/2}_*.
$$
That this is possible follows from the dominated convergence theorem and the following consequence of \eqref{E:Poho}:
\begin{align*}
F^{1/2}((1-\eps)Q) = \bigl[1-4\eps^2+O(\eps^3)\bigr] F^{1/2}_* \quad\text{for $\eps$ small enough}.
\end{align*}

By Proposition~\ref{coer}, there exists a unique global solution $u_n:\R\times\Omega\to \C$ to \eqref{nls} with initial data $u_n(0)=\phi_n$.
We will prove that for $n$ sufficiently large,
\begin{align*}
\tilde u_n(t,x):= e^{it} \phi_n(x)
\end{align*}
is an approximate solution to \eqref{nls} on any fixed compact time interval $[-T,T]$.  More precisely, we will show that for a fixed $T>0$,
\begin{align}\label{blowup approx}
\lim_{n\to \infty}\bigl\| (i\partial_t + \Delta_\Omega) \tilde u_n + |\tilde u_n|^2\tilde u_n\bigr\|_{N^1([-T,T])}=0.
\end{align}
Therefore, by Lemma~\ref{lm:stability}, for $n$ sufficiently large we get
$$
\|u_n\|_{L_{t,x}^5([-T,T]\times\Omega)} \gtrsim \|\tilde u_n\|_{L_{t,x}^5([-T,T]\times\Omega)} \gtrsim_Q T.
$$
As $T>0$ is arbitrary, this gives the desired blowup of the $L^5_{t,x}$-norm as $n\to \infty$.

It remains to prove \eqref{blowup approx}.  A simple computation shows that
\begin{align*}
&\bigl[(i\partial_t + \Delta_\Omega)\tilde u_n + |\tilde u_n|^2\tilde u_n\bigr](x)\\
&= e^{it}(1-\eps_n)\bigl[ (\chi^3-\chi)(x) Q(x+x_n) +2\nabla \chi(x)\nabla Q(x+x_n) + \Delta\chi(x) Q(x+x_n)  \bigr]\\
&\quad -e^{it}\eps_n(1-\eps_n)(2-\eps_n)\chi^3(x) Q(x+x_n).
\end{align*}
By the dominated convergence theorem, we see that the $N^1([-T,T])$-norm of the sum in the square brackets converges to zero as $|x_n|\to \infty$.  Taking $\eps_n\to 0$, one can also render arbitrarily small the $N^1([-T,T])$-norm of the last term on the right-hand side of the equality above.  This proves \eqref{blowup approx} and so completes the proof of Proposition~\ref{P:blowup}.\qed

\end{document}